\documentclass[11pt]{article}
\usepackage{graphicx} 

\usepackage{epsf,epsfig,amsfonts,amsgen,amsmath,amstext,amsbsy,amsopn,amsthm,cases,listings,color
}
\usepackage{ebezier,eepic}
\usepackage{bbm}
\usepackage{color}
\usepackage{multirow}
\usepackage{epstopdf}
\usepackage{pgf,tikz}
\usepackage{mathrsfs}
\usepackage[marginal]{footmisc}
\usepackage{enumitem}
\usepackage[titletoc]{appendix}
\usepackage{booktabs}
\usepackage{url}
\usepackage{subfigure}
\usepackage{mathrsfs}
\usepackage{comment}

\usetikzlibrary{arrows}

\allowdisplaybreaks[1]

\definecolor{uuuuuu}{rgb}{0.27,0.27,0.27}
\definecolor{sqsqsq}{rgb}{0.1255,0.1255,0.1255}

\newtheorem{definition}{Definition} 
\newtheorem{theorem}[definition]{Theorem}
\newtheorem{lemma}[definition]{Lemma}

\newcommand{\km}{K_4^-}

\newenvironment{pf}{\noindent{\bf Proof}}

\newcommand{\dm}[1]{\textcolor{blue}{\textbf{[DM: } #1\textbf{]}}}

\usepackage[utf8]{inputenc}
\title{$\km$-free triple systems without large stars in the complement}
\author{
Dhruv Mubayi\thanks{Department of Mathematics, Statistics and Computer Science, University of Illinois, Chicago, IL 60607. Email: mubayi@uic.edu. Research partially supported by NSF Awards DMS-1952786 and DMS-2153576.} \and
Nicholas Spanier\thanks{Department of Mathematics Statistics and Computer Science, University of Illinois, Chicago, IL 60607.
			Email: spanier2@uic.edu. Research partially supported by NSF Awards DMS-1952786 and DMS-2153576.
			  }}

\date{\today}

\begin{document}

\maketitle

\begin{abstract}
    The $n$-star $S_n$ is the $n$-vertex triple system with ${n-1 \choose 2}$ edges all of which contain a fixed vertex, and $\km$ is the unique triple system with four vertices and three edges.  We prove that the Ramsey number $r(K_4^-, S_n)$ has order of magnitude $n^2 /\log n$.
    
    This  confirms a conjecture of Conlon, Fox, He, Suk, Verstra\"ete and the first author. It also generalizes the well-known bound of Kim for the graph Ramsey number $r(3,n)$, as the link of any vertex in a $K_4^-$-free triple system is a triangle-free graph. Our method builds on the approach of Guo and Warnke who adapted Kim's lower bound for $r(3,n)$ to the pseudorandom setting. 
\end{abstract}
\section{Introduction}
A $k$-uniform hypergraph $H$ (henceforth $k$-graph) is a pair $(V(H), E(H))$
 where $E(H)$ is a collection of $k$-element subsets of $V(H)$. We often associate $H$ with its edge set $E(H)$, when $V(H)$ is obvious from context. Write $K_n^k$ for the complete $k$-graph  on $n$ vertices. Given $k$-graphs $G$ and $H$, the Ramsey number $r(G, H)$ is the minimum $N$ such that every red/blue coloring of $E(K_N^k)$ has a monochromatic red copy of $G$ or a monochromatic blue copy of $H$; when $H = K_n^k$, we simply write $r(G, n)$. When $G$ is fixed and $n \rightarrow \infty$, we often refer to this class of parameters as off-diagonal Ramsey numbers. 

The off-diagonal graph Ramsey numbers $r(s,n)=r(K_s, K_n)$ have been intensively studied for many decades.
 Erd\H os conjectured that for fixed $s \ge 3$, we have $r(s,n) = n^{s-1+o(1)}$ and this was reiterated recently by the first author and Verstra\"ete~\cite{MV}. Only the first two cases, $s=3,4$ have been proved. The order of magnitude of $r(3,n)$ is $n^2/\log n$, where the upper bound was obtained by  Ajtai-Koml\'os-Szemer\'edi~\cite{AKS} and the lower bound was proved by Kim~\cite{Kim} and later by Bohman ~\cite{Bohman} using a different method. More recently, Mattheus and Verstra\"ete proved that $r(4,n) = n^{3-o(1)}$ with the order of magnitude still undetermined. In this paper we consider an analogous parameter to $r(3,n)$ in the hypergraph setting.

It is well-known that for each $s \ge 4$, the hypergraph Ramsey number $r(K_4^3, K_n^3) > 2^{cn}$ for some positive constant $c$. On the other hand, there is a well-known large class $\mathcal{C}$ of 3-graphs $G$ for which it is known that $r(G, n)$ is polynomial in $n$. Recently, it was conjectured~\cite{CFGHMSVY} that $r(G, n)$ is at most polynomial in $n$ if and only if $G \in \mathcal{C}$. On the other hand, it is not clear whether $G \not\in \mathcal{C}$ implies that $r(G, n)$ is exponential in $n$. In fact, the following result from~\cite{CFHMSV} showed an intermediate growth rate for the following related problem.  
The link graph of a vertex $v$ in a 3-graph $H$ is the graph whose vertex set is $V(H)\setminus \{v\}$ and edge set is $\{yz: vyz \in E(H)\}$. Write $S_n$ for the $n$-vertex star: this is the 3-graph with vertex set $[n]:=\{1, \ldots, n\}$ and edge set $\{e \in {[n] \choose 3}: 1 \in e\}$. Alternatively, $S_n$ is the 3-graph for which there is a vertex whose link graph is isomorphic to $K_{n-1}$.  The 3-graph $S_n$ is a natural way to generalize $K_n$ to hypergraphs.   It is proved in~\cite{CFHMSV} that there are positive constants $c, c'$ such that 
\begin{equation} \label{eqn:k4sn} 2^{c \log^2 n} < r(K_4^3, S_n) < 2^{c' n^{2/3} \log n}.\end{equation}
This shows that $r(K_4^3, S_n)$ is  already superpolynomial in $n$. Therefore, one needs to consider sparser 3-graphs $G$ to detect polynomial growth rates for $r(G, S_n)$. The smallest such nontrivial 3-graph $G$ is $\km$, the unique 3-graph with four vertices and three edges (if $G$ is the 3-graph on four vertices and two edges, then it is trivial to see that $r(G, S_n) = \Theta(n)$). 

The link graph $L$ of a vertex in an $N$-vertex $\km$-free 3-graph $H$ is a triangle-free graph  on $N-1$ vertices, and hence by the result $r(3,n) = O( n^2/\log n)$ of~\cite{AKS}, the independence number of $L$ is at least $\Omega(\sqrt{N \log N})$. Consequently, the complement of every $\km$-free 3-graph on $N$ vertices contains an $n$-vertex star with $n =\Omega(\sqrt{N \log N})$. This shows
$$r(\km, S_n)  = O\left(\frac{n^2}{\log n}\right).$$
In the other direction, it was proved in~\cite{CFHMSV} using the local lemma that 
$r(\km, S_n)  = \Omega(n^2/\log^2 n)$. 
This led the authors of~\cite{CFHMSV} to conjecture that 
\begin{equation} \label{eqn:conj}r(\km, S_n)  = \Theta\left(\frac{n^2}{\log n}\right). \end{equation}
A similar state of affairs existed for the graph case $r(3,n)$ prior to Kim's improved lower bound. Note, however, that (\ref{eqn:conj}) is substantially stronger than Kim's result $r(3,n) = \Omega(n^2/\log n)$ as it posits the existence of a triple system such that {\em every} link graph is an optimal (in order of magnitude)   $r(3,n)$ graph. Also, as indicated in (\ref{eqn:k4sn}), hypergraph problems sometimes display novel phenomena compared to what we typically witness for graphs, so it was by no means a forgone conclusion that (\ref{eqn:conj}) should hold similar to the graph result $r(3,n)=\Theta(n^2/\log n)$.
In this paper, we prove (\ref{eqn:conj}).

\begin{theorem}
     \label{mainSn}
There is an absolute constant $c$ such that for all sufficiently large $N$, there is a $\km$-free 3-graph on $N$ vertices whose complement contains no copy of $S_n$, where $n = c\sqrt{N\log N}$. In other words, $r(\km, S_n) = \Omega(n^2/\log n)$.
\end{theorem}

Our theorem follows from the following stronger result, which guarantees not only that for every vertex $x$ and every $n$-set $A$ omitting $x$, there is an edge of the form $xyz$, where $yz \in {A \choose 2}$, but in fact that there are many such edges. It is more convenient to consider the bipartite setting as follows.
Given a $3$-graph $H$, disjoint sets  $A, B \subseteq V(H)$, and $x \in V(H) \setminus (A \cup B)$, let $e_{H, x}(A, B)$  be the number of edges in $H$ of the form $xab$ where $a \in A$ and $b \in B$. The following theorem is our main result.

\begin{theorem} [\bf Main result]
\label{main_theorem}
There exists $\beta_0, D_0$ such that for all $\delta \in (0,1], \beta \in (0, \beta_0)$,  $C > \frac{D_0}{\delta^2 \sqrt{\beta}}$ and  $N > N_0(\delta, \beta, C)$, there exists a $K_4^-$ free $3$-graph $H$ on $N$ vertices such that for every two disjoint $n$-vertex subsets $A, B$ with $n := C\sqrt{N \log N}$ and for every vertex $x \not\in A \cup B$, $$e_{H, x}(A, B) = (1 \pm \delta) \rho n^2$$ where $\rho := \sqrt{\frac{\beta \log N}{3N}}$.
\end{theorem}

There are two general approaches to obtain graphs or hypergraphs that have good Ramsey properties of the type sought in Theorems~\ref{mainSn} and~\ref{main_theorem}: the nibble method pioneered by Ajtai-Koml\'os-Szemer\' edi~\cite{AKS, AKS81} and R\"odl~\cite{Rodl85}, and the differential equations method developed by Wormald~\cite{Wormald} for combinatorial settings. We prove Theorem~\ref{main_theorem} using the nibble method. This was the approach used by Kim~\cite{Kim} in his proof of $r(3,n) = \Omega(n^2/\log n)$.  More recently, Guo and Warnke~\cite{GW} established a stronger result than $r(3,n) = \Omega(n^2/\log n)$ in two ways. First, they not only found one triangle-free graph on $N$ vertices with independence number $O(\sqrt{N \log N})$, but in fact decomposed most of the edges of $K_N$ into such graphs. Second, they were able to obtain such a decomposition as long as their initial graph was sufficiently pseudorandom. Along the way, they introduced several simplifications of Kim's argument.  

Our proof of Theorem~\ref{main_theorem} follows the broad outline of Kim's approach together with the additional technical ingredients provided by Guo and Warnke. However, creating $\km$-free 3-graphs without large stars in their complement, rather than just triangle-free graphs with small independent sets, is more challenging. For example, certain random variables that we must track that appear in  our semirandom algorithm do not even make sense to define in the graph case. New ideas are required to implement the additional technical steps needed to control their deviation probability from their mean (see, for example, the event $\mathcal{N}_i^+$, or the Proof of Lemma~\ref{Q event lemma} in Section~\ref{ss}).

We end by remarking that we were not able to prove Theorem~\ref{main_theorem} using the random greedy $K_4^-$-free   process, and this remains open.
\section{Nibble Algorithm}

Throughout the paper we use the notation $xy$ and $xyz$ for  $\{x,y\}$ and $\{x,y,z\}$, respectively. 

We will form a $\km$-free $3$-graph on an $N$-element set $V$ by iteratively selecting a set of edges and removing a subset of those edges to destroy all copies of $\km$.
At each step in the process, we will maintain a triple $(E_i, H_i, O_i)$ where
\begin{itemize}
     \item $E_i$ is
    a set of {\em picked} or {\em chosen} edges which may contain copies of $\km$, 
    
    \item  $H_i \subseteq E_i$ where  $H_i$ is $K_4^-$-free. We will iterate the process $I$ times and our final 3-graph will be $H_I$.

     \item  $O_i \subseteq {V\choose 3}\setminus E_i$ is a set of {\em open} edges. No edge in $O_i$ forms a $\km$ together with any two edges from $E_i$. 
\end{itemize}

We will start by initializing  $E_0 = H_0 = \emptyset$ and $O_0 = \binom{V}{3}$.
Assume that at step $i$ we have sets $E_i, H_i, O_i$ such that $H_i \subseteq E_i$, $E_i \cap O_i = \emptyset$, $H_i$ is $K_4^-$ free, and for any $e \in O_i$ there do not exist $f, g \in E_i$ such that $efg$ forms a $K_4^-$. Set
\begin{equation}
\label{eqn: sigma and p definition}
\sigma := \frac{1}{\log^2 N} \qquad \hbox{ and } \qquad p := \frac{\sigma}{\sqrt{N}} = \frac{1}{\sqrt N\log^2 N }.
\end{equation}
Form a set of edges $\Gamma_{i+1} \subseteq O_i$ by including each edge $e \in O_i$ in $\Gamma_{i+1}$ independently with probability $p$.
Let $E_{i+1} = E_i \cup \Gamma_{i+1}$. Because $\Gamma_{i+1} \subseteq O_i$, the set $E_{i+1}$ contains no $K_4^-$ with exactly one edge in $\Gamma_{i+1}$.

There may be three edges in $\Gamma_{i+1}$ which form a $K_4^-$, or two edges in $\Gamma_{i+1}$ and one edge in  $E_i$ which together form a $K_4^-$.  We gather these bad edges and remove a cleverly chosen subset of them so that there is
no $\km$ of these two types after we have removed the edges. Then we add to $H_i$ the edges that remain and form $H_{i+1}$. This process will ensure that
$H_{i+1}$ is $K_4^-$ free. We do this as follows. First, let
 \begin{equation}
 \label{eqn: B^2 def}
 \notag
 B^2_{i+1} = \{ef \subseteq \Gamma_{i+1} : \exists g \in H_i, efg \text{ is a }K_4^-\}
\notag \end{equation}
 \begin{equation}
 \label{eqn: B^3 def}
 \notag
B^3_{i+1} = \{efg \subseteq \Gamma_{i+1} : efg \text{ is a } K_4^-\}
\end{equation}
\begin{equation}
\label{eqn: B def}
\notag
B_{i+1} = B^2_{i+1} \cup B^3_{i+1}.
\end{equation}

Next, let $D_{i+1}$ be a maximal collection of pairwise disjoint subsets of
$B_{i+1}$. More precisely $D_{i+1} \subseteq B_{i+1}$ with the  property that for any $ef \in B^2_{i+1}$, at least  one of $e,f$ lies in some set in $D_{i+1}$, and 
for any $e'f'g' \in B^3_{i+1}$,  at least one of $e',f',g'$ lies in some set in $D_{i+1}$. Let 
\begin{equation}
\notag 
H_{i+1} = H_i \cup (\Gamma_{i+1} \setminus \cup_{S \in D_{i+1}} \cup_{e \in S} \{e\}).
\end{equation}
In other words, we have added to $H_i$ all sets that do not appear in an element of $D_{i+1}$.

Let us now prove that $H_{i+1}$ is $\km$-free. We may assume (by induction) that $H_i$ is $\km$-free. There is no $\km$ in $H_{i+1}$ with exactly two edges in $H_i$ by the definition of $O_i$. There is no $\km$ in $H_{i+1}$ with exactly one edge $g$ in $H_i$ since the other two edges $e,f$ must satisfy $ef \in B^2_{i+1}$ so either $e$ or $f$ must lie in some pair from $D_{i+1}$ which means it is not in $H_{i+1}$. Lastly, there is no $\km$ in $H_{i+1}$ with all three edges $e,f,g$ in $H_{i+1}\setminus H_i$ since  $e,f,g$ must satisfy $efg \in B^3_{i+1}$ so one of them must lie in some triple from $D_{i+1}$ which means it is not in $H_{i+1}$.

We now explain how the set $O_i$ is updated to $O_{i+1}$. The updated set $O_{i+1}$ must have the property that each edge in $O_{i+1}$ does not form a copy of $\km$ with already chosen edges, so that it will be possible to choose it in step $i+1$; this copy of $\km$ can have both its remaining edges chosen in step $i$, or have one edge chosen in step $i$ and one edge chosen in some previous step. Formally, for each $e \in O_i$, let 
\begin{equation}
\label{eqn: S_e def}
\notag
\hat{S}_{i}(e) = \{f \in O_i : \exists g \in E_i, efg \text{ is a } K_4^-\}.
\end{equation}
Next, define
\begin{equation}
\label{eqn: C^1 def}
\notag
C^1_{i+1} = \{e \in O_i : \hat{S}_i(e) \cap \Gamma_{i+1} \neq \emptyset\}
\end{equation}
\begin{equation}
\label{eqn: C^2 def}
\notag
C^2_{i+1} = \{e \in O_i : \exists f, g \in \Gamma_{i+1}, efg \text{ is a } K_4^-\}.
\end{equation}
Our rough plan is to update $O_i$ by removing all
triples from $\Gamma_{i+1} \cup C^1_{i+1} \cup C^2_{i+1}$.  However, for technical reasons, it is convenient to remove some more edges from $O_i$ as follows. Let $Y_{i+1}\subseteq O_i$ be the (random) subset obtained by including each 
 $e \in O_i$ in $Y_{i+1}$ independently with probability
\begin{equation}
    \label{eqn: p hat def}
    \hat{p}_{e,i} := 1 - (1-p)^{\max \{6\sqrt{N}q_i (\pi_i + \sqrt{\sigma}) - |\hat{S}_i(e)|, 0\}}.
    \end{equation}
We form $O_{i+1}$ by removing from   $O_i$ the edges which were picked in $\Gamma_{i+1}$, the edges in $C_{i+1}^1$ and $C_{i+1}^2$, and the edges in $Y_{i+1}$:
\begin{equation}
    \label{eqn: O_i change def}
    O_{i+1} = O_i \setminus (\Gamma_{i+1} \cup C^1_{i+1} \cup C^2_{i+1} \cup Y_{i+1}).
\end{equation}

We will continue iterating this process until we have formed the set $H_I$ where $I := \lceil N^{\beta} \rceil$. Now let $H := H_I$. Our goal is to prove that whp (with high probability) $H$ satisfies the conditions in Theorem \ref{main_theorem}.

\section{Expected Trajectories via Differential Equations} \label{sec:traj}

We use the following heuristics to predict the behavior of the size of our sets $O_i$ and $E_i$. Assume that for all  $e \in {V \choose 3}$, we have 
$$\mathbb{P}(e \in O_i) \approx q_i \qquad \hbox{ and } \qquad \mathbb{P}(e \in E_i) \approx \frac{\pi_i}{\sqrt{N}},$$ where $\pi_i=O(\sqrt{\log N})$ and these events are approximately independent for all $e$. Now notice that
\begin{equation}
    \label{eqn: One Step Change E_i heuristic}
\mathbb{P}(e \in E_{i+1}) - \mathbb{P}(e \in E_i) = \mathbb{P}(e \in \Gamma_{i+1} | e \in O_i) \cdot \mathbb{P}(e \in O_i) \approx p q_i =\frac{\sigma q_i}{\sqrt{N}}.
\end{equation}
Then multiplying (\ref{eqn: One Step Change E_i heuristic}) by $\sqrt{N}$ gives $\pi_{i+1} - \pi_i \approx \sigma q_i$. Also since $E_0 = \emptyset$, we have $\pi_0 \approx 0$.

Next, we approximate $O_i \setminus O_{i+1} = \Gamma_{i+1} \cup C^1_{i+1} \cup C^2_{i+1} \cup Y_{i+1} \approx C^1_{i+1} \cup Y_{i+1}$. Notice 
\begin{equation}
\label{eqn: Expected Value of S_e}
\mathbb{E}(|\hat{S}_i(e)|) \approx 3(N-3)q_i \left(1 - \left(1 - \frac{\pi_i}{\sqrt{N}}\right)^2\right) \approx 3Nq_i\left(\frac{2\pi_i}{\sqrt N}\right)
\approx 6\sqrt{N}q_i \pi_i.
\end{equation}
Thus by the estimate for $O_i \setminus O_{i+1}$ and from (\ref{eqn: p hat def}) and (\ref{eqn: Expected Value of S_e}), we get that for all $e \in O_i$ 
\begin{equation}
\label{eqn: Probability e in O_i+1 given O_i heuristic}
\notag
    \mathbb{P}(e \in O_{i+1} | e \in O_i) \approx (1 - \hat{p}_{e, i})(1- p)^{1 + |\hat{S}_{i}(e)|} \approx (1-p)^{6\sqrt{N}q_i (\pi_i + \sqrt{\sigma})} \approx 1 - 6q_i \pi_i \sigma.
    \end{equation}
This heuristic suggests that  $\mathbb{P}(e \in O_{i+1} | e \in O_i)$ depends only on the step $i$ of our iteration, and this is the reason we introduced the stabilization probability  $\hat{p}_{e, i}$. 
Therefore, 
\begin{equation}
    \label{eqn: change q heuristic}
    \notag
q_{i+1} - q_i \approx \mathbb{P}(e \in O_{i+1}) - \mathbb{P}(e \in O_i) = \mathbb{P}(e \in O_{i+1} | e \in O_i)\mathbb{P}(e \in O_i) - \mathbb{P}(e \in O_i) \approx -6q_i^2\pi_i \sigma.
\end{equation}

Next suppose $\pi_i \approx \Psi(i\sigma)$ where $\Psi$ is a smooth function. Then 
$$\Psi'(i\sigma) \approx \frac{\Psi((i+1)\sigma) - \Psi(i\sigma)}{\sigma} \approx
\frac{\pi_{i+1}-\pi_i}{\sigma} \approx q_i$$
and for $x=i\sigma$, 
$$\Psi''(x) \approx \Psi''(i\sigma) = \frac{\Psi'((i+1)\sigma) - \Psi'(i\sigma)}{\sigma} \approx \frac{q_{i+1} - q_i}{\sigma} \approx -6q_i^2\pi_i \approx 
-6(\Psi'(x))^2\Psi(x).$$ Solving this differential equation yields
\begin{equation}
\label{eqn: Psi differential equation}
\Psi'(x) = e^{-3\Psi^2(x)}.
\end{equation}
Moreover, writing $f= \Psi^{-1}$ so that $(f\Psi)(x) = x$ for $x\ge 0$, and taking derivatives we obtain
$$1 = (f\Psi)'(x) = f'(\Psi(x)) \Psi'(x) = f'(\Psi(x)) e^{-3\Psi^2(x)}.$$
Writing $t = \Psi(x)$, we get $f'(t) = e^{3t^2}$.  Further, since $\pi_0 \approx 0$ and $q_0 \approx 1$, we set  $\Psi(0) = 0$ and $\Psi'(0) = 1$. Consequently,
\begin{equation} \label{eqn:implicit}
x = f(\Psi(x))- f(\Psi(0)) = \int_{\Psi(0)}^{\Psi(x)} f'(t) \, dt = \int_0^{\Psi(x)} e^{3t^2}\, dt.    
\end{equation}

Now we will formally define $\Psi(x)$ as the solution to the differential equation (\ref{eqn: Psi differential equation}) with $\Psi(0) = 0$. Then (\ref{eqn:implicit}) also holds  and we define 
\begin{equation}
\label{eqn: q def}
q_i = \Psi'(i\sigma)
\end{equation}
\begin{equation}
\label{eqn: pi def}
\pi_i = \sigma + \sum_{j=0}^{i-1} \sigma q_j.
\end{equation}

\section{Technical Estimates}

In this section will collect some bounds on $q_i$ and $\pi_i$ which we will use throughout the proof of the main theorem. The proofs of these bounds will be given in the Appendix. Recall that 
$$\sigma= \frac{1}{\log^2 N}, \qquad \beta \in (0, \beta_0), \qquad I=\lceil N^{\beta}\rceil.$$

\begin{lemma}
    \label{q and pi bounds}
    Let $\Psi(x), q_i$, and $\pi_i$  be defined as in (\ref{eqn:implicit}), (\ref{eqn: q def}), and (\ref{eqn: pi def}). 
        \begin{equation}
        \label{Psi bound}
        \sqrt{\frac{\log (\sqrt{3} x)}{3}} - \frac{1}{\sqrt{3}} \leq \Psi(x) \leq \sqrt{\frac{\log (\sqrt{3}x)}{3}} + \frac{1}{\sqrt{3}} \text{ for } x \geq e
        \end{equation}

        \begin{equation}
        \label{q less than 1}
        0 \leq q_i \leq q_0 = 1 \text{ for } i \geq 0  
        \end{equation}

        \begin{equation}
        \label{q and pi bounds pi close to psi}
        \pi_{i+1} - \pi_i = \sigma q_i \text{ and } \pi_i - \Psi(i\sigma) \in [\sigma, 2\sigma] \text{ for all } i \geq 0 
        \end{equation}

        \begin{equation}
        \label{sqrt sigma pi < 1}
        \sqrt{\sigma} \pi_i \leq 1 \text{ for } 0 \leq i \leq I
        \end{equation}

        \begin{equation}
        \label{q and pi bounds, q times pi less than 1}
            q_i \pi_i^k \leq 1 \text{ for } 0 \leq i \leq I \text{ and } k \in \{1,2\}
        \end{equation}

        \begin{equation}
        \label{bound on q_i - q_i+1}
        |(q_i - q_{i+1}) - 6 \sigma q_i^2\pi_i| \leq 16\sigma^2 q_i^2 \text{ for all } i \geq 0
        \end{equation}

        \begin{equation} \label{eqn:sigmaqN}
        q_i \geq \frac{1}{10}N^{-\beta + o(1)} \text{ for all } 0 \leq i \leq I
        \end{equation}

        \begin{equation}
        \label{p hat < q}
        \hat{p}_{e,i} \leq q_i \text{ for all }e \in O_i
        \end{equation}

        \begin{equation}
        \label{q_i - q_i+1 < sigma min...}
        0 \leq q_i - q_{i+1} \leq 12\sigma \min \{q_i, q_{i+1}, q_i \pi_i \}
        \end{equation}
\end{lemma}

In particular, if we let $\beta_0 = \frac{1}{100}$, then by (\ref{eqn:sigmaqN}) and $\beta \leq \beta_0$, there exists $c>0$ such that the following holds for all $0\le i\le I$ and $j <10$:

\begin{equation}
\label{q_i sqrt N estimate}
q_i^j \sqrt{N} \geq N^{-j\beta + \frac{1}{4}} \geq N^c.
\end{equation}

\section{Events}

Let $X_i = (E_i, H_i, O_i, \Gamma_i, Y_i)$
and  $X = \{X_i\}_{0 \leq i \leq I}$. We let $\mathcal{F}_i$ denote the $\sigma$-algebra generated by $\{X_j\}_{0 \leq j \leq i}$ and $(\mathcal{F}_i)_{0 \leq i \leq I}$ be the natural filtration with $\{X_i\}_{0 \leq i \leq I}$. Recall that we use the notation $abc$ for $\{a,b,c\}$. 

Next we give definitions which apply to  an edge $xuv$ with a distinguished vertex $x$.
\begin{equation}
\label{eqn: R_x,uv def}
\notag
R_{i}(x,uv) := \{w \in [N] \setminus \{u,v,x\} : xuw, xvw \in O_i\}
\end{equation}
\begin{equation}
    \label{eqn: S_x,uv def}
    \notag
    S_{i}(x,uv) = \{w \in [N] \setminus \{u,v,w\} : |\{xuw,xvw\} \cap O_i| = |\{xuw,xvw\} \cap E_i| = 1\}
\end{equation}
\begin{equation}
\label{eqn: T_x,uv def}
\notag
T_{i}(x,uv) := \{w \in [N] \setminus \{u,v,x\} : xuw, xvw \in E(i)\}
\end{equation}
\begin{equation}
    \notag
    U_i(x,u,v,w) := \{z \in [N] \setminus \{x,u,v,w\} : xuv, xuw, xuz \in O_i \text{ and } xvz, uwz \in E_i\}
\end{equation}
\begin{equation}
    \notag
    \hat{S}_i(x,uv) = \{xuw \in O_i : w \in S_i(x,uv)\}
\end{equation}
The following figures depict the above definitions. Note for $S_i(x,uv)$ that there are two types of edges, one where we include $xuw \in O_i$ and $xvw \in E_i$ and the other where we have $xuw \in E_i$ and $xvw \in O_i$.

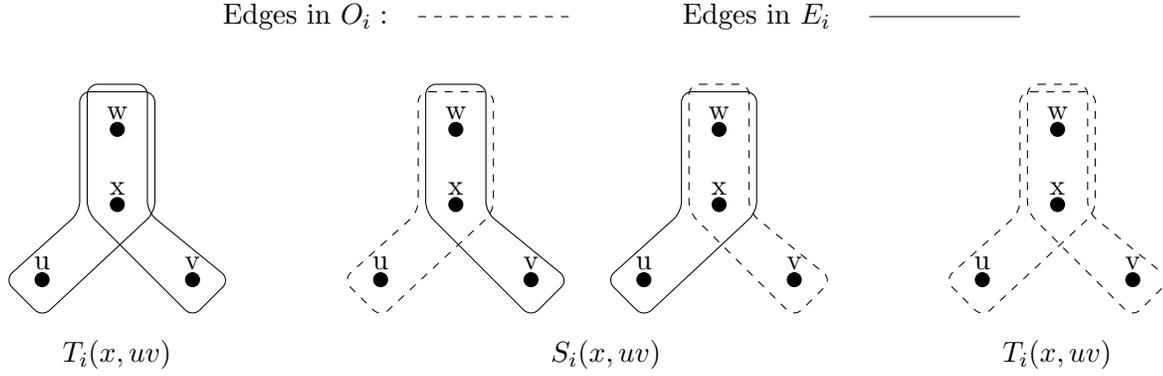
\begin{figure}[hbt!]
   \begin{tikzpicture}
       \node(T) at (0,-2){$T_i(x,uv)$};
       \node (x) at (0,0){};
       \fill(x) circle(0.1) node[above] {x};
       \node (w) at (0,1){};
       \fill(w) circle(0.1) node[above] {w};
       \node (u) at (-1,-1){};
       \fill(u) circle(0.1) node[above] {u};
       \node (v) at (1,-1){};
       \fill(v) circle(0.1) node[above] {v};

       \draw[rounded corners] (-.5,1.5) -- (.5,1.5) -- (.5,-.1) -- (-1,-1.5) -- (-1.5,-1) -- (-.5,-.1) -- cycle;

       \draw[rounded corners] (-.4, 1.6) -- (.4,1.6) -- (.4,-.1) -- (1.5,-1) -- (1,-1.5) -- (-.4,-.1) -- cycle;

\node(S) at (6.5,-2){$S_i(x,uv)$};
       \node (x) at (0+4.5,0){};
       \fill(x) circle(0.1) node[above] {x};
       \node (w) at (0+4.5,1){};
       \fill(w) circle(0.1) node[above] {w};
       \node (u) at (-1+4.5,-1){};
       \fill(u) circle(0.1) node[above] {u};
       \node (v) at (1+4.5,-1){};
       \fill(v) circle(0.1) node[above] {v};

       \draw[rounded corners, dashed] (-.5+4.5,1.5) -- (.5+4.5,1.5) -- (.5+4.5,-.1) -- (-1+4.5,-1.5) -- (-1.5+4.5,-1) -- (-.5+4.5,-.1) -- cycle;

       \draw[rounded corners] (-.4+4.5, 1.6) -- (.4+4.5,1.6) -- (.4+4.5,-.1) -- (1.5+4.5,-1) -- (1+4.5,-1.5) -- (-.4+4.5,-.1) -- cycle;

       \node (x) at (0+8,0){};
       \fill(x) circle(0.1) node[above] {x};
       \node (w) at (0+8,1){};
       \fill(w) circle(0.1) node[above] {w};
       \node (u) at (-1+8,-1){};
       \fill(u) circle(0.1) node[above] {u};
       \node (v) at (1+8,-1){};
       \fill(v) circle(0.1) node[above] {v};

       \draw[rounded corners] (-.5+8,1.5) -- (.5+8,1.5) -- (.5+8,-.1) -- (-1+8,-1.5) -- (-1.5+8,-1) -- (-.5+8,-.1) -- cycle;

       \draw[rounded corners, dashed] (-.4+8, 1.6) -- (.4+8,1.6) -- (.4+8,-.1) -- (1.5+8,-1) -- (1+8,-1.5) -- (-.4+8,-.1) -- cycle;

\node(R) at (12.5,-2){$T_i(x,uv)$};
       \node (x) at (0+12.5,0){};
       \fill(x) circle(0.1) node[above] {x};
       \node (w) at (0+12.5,1){};
       \fill(w) circle(0.1) node[above] {w};
       \node (u) at (-1+12.5,-1){};
       \fill(u) circle(0.1) node[above] {u};
       \node (v) at (1+12.5,-1){};
       \fill(v) circle(0.1) node[above] {v};

       \draw[rounded corners, dashed] (-.5+12.5,1.5) -- (.5+12.5,1.5) -- (.5+12.5,-.1) -- (-1+12.5,-1.5) -- (-1.5+12.5,-1) -- (-.5+12.5,-.1) -- cycle;

       \draw[rounded corners, dashed] (-.4+12.5, 1.6) -- (.4+12.5,1.6) -- (.4+12.5,-.1) -- (1.5+12.5,-1) -- (1+12.5,-1.5) -- (-.4+12.5,-.1) -- cycle;

\node(OpenKey) at (2.5,2.5) {Edges in $O_i:$};
\node(PickedKey) at (8.5,2.5) {Edges in $E_i$};
\draw[dashed] (4,2.5) -- (6,2.5);
\draw (10,2.5) -- (12,2.5);
       
   \end{tikzpicture}

\caption{$T_i(x,uv) \text{, } S_i(x,uv) \text{, and } R_i(x,uv)$}
\label{T S and R fig} 
\end{figure}

\begin{figure}[hbt!]
\centering
\begin{tikzpicture}
\node (x) at (0,0) {};
\fill(x) circle(0.1) node[above]{$x$};
\node (v) at (-2,2) {};
\fill(v) circle(0.1) node[above]{$v$};
\node(u) at (2,2) {};
\fill(u) circle(0.1) node[above]{$u$};
\node (z) at (-2,-2){};
\fill(z) circle(0.1) node[below]{$z$};
\node (w copy) at (2,-2) {};
\fill(w copy) circle(0.1) node[below]{$w_3$};

\draw [rounded corners,dashed] (-2.6,2.4) -- (2.6,2.4) -- (0,-0.8) -- cycle;

\draw [rounded corners,dashed] (2.4,2.7) -- (-0.4,0) -- (2.4,-2.7) -- cycle;

\draw [rounded corners] (-2.4,2.7) -- (0.4,0) -- (-2.4,-2.7) -- cycle;

\draw[rounded corners] (-2.5,-1.7) -- (1.8, -1.7) -- (1.8, 2.7) -- (2.5, 2.7) -- (2.5, -2.5) -- (-2.5, -2.5) -- cycle;

\draw[rounded corners, dashed] (-3,-2) -- (-2,-3) -- (3,2) -- (2,3) -- cycle;

\node(OpenKey) at (-4,3.5) {Edges in $O_i$};
\node(PickedKey) at (2,3.5) {Edges in $E_i$};
\draw[dashed] (-2.5,3.5) -- (-0.5,3.5);
\draw (3.5,3.5) -- (5.5, 3.5);
    
\end{tikzpicture}
\caption{$U_i(x,u,v,w)$}
\label{U fig}
\end{figure}
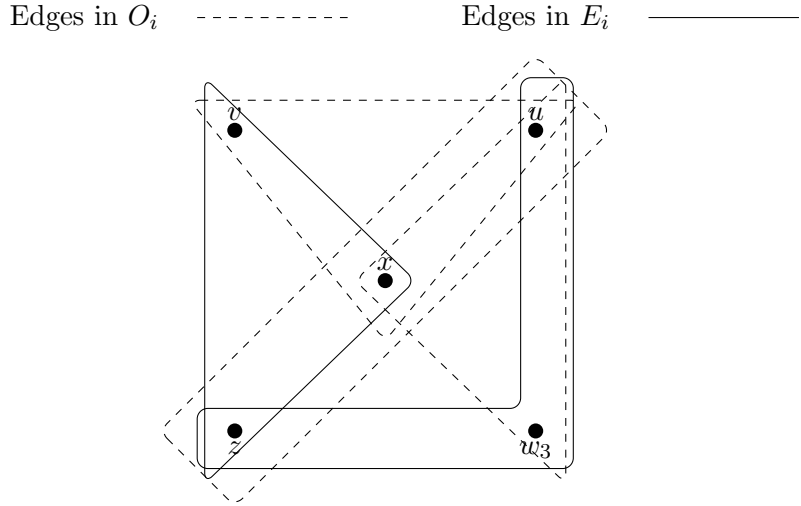

Note that $S_i(x,uv)$ denotes a set of vertices which have one open edge and one closed edge with the pairs $xu$ and $xv$, and $\hat{S}_i(x,uv)$ describes the corresponding set of open edges. Thus $|S_i(x,uv)| = |\hat{S}_i(x,uv)|$.

Given a set $F$ of edges, and vertices $u,v$, let $N_F(uv)=\{w: uvw \in F\}$ be the set of vertices that form an edge of $F$ with $u$ and $v$. Also given a vertex $x$, two vertex subsets $A, B \subseteq [N] \setminus \{x\}$ and a set $F$ of edges, let $F(A,B,x) := \{xuv \in F : u \in A \text{ and } v \in B\}$. 
Define the events
\begin{equation}
\label{eqn: event N_i}
    \mathcal{N}_i = \{|N_{O_i}(vx)| \leq q_i N \text{ and } |N_{\Gamma_i}(vx)| \leq 2q_{i-1}\sigma \sqrt{N} \text{ for all } x, v \in [N]\}
\end{equation}
\begin{equation}
\label{eqn: event P_i}
\mathcal{P}_i = \{|R_{i}(x,uv)| \leq q_i^2 N \text{, }|S_{i}(x,uv)| \leq 2q_i\pi_i\sqrt{N}\text{, }|T_{i}(x,uv)| \leq i (\log N)^9 \text{ for all } x,u,v \in [N]\}
\end{equation}
\begin{equation}
    \label{eqn: event P_i^+}
    \mathcal{P}_i^+ = \{U_i(x,u,v,w) \leq i (\log N)^{9} \text{ for all } x,u,v,w \in [N]\}
\end{equation}

By (\ref{Psi bound}), $\Psi^2(x) \sim (1/3)\log x$ as $x \rightarrow \infty$. Hence $$q_I = \Psi'(I\sigma) = e^{-3\Psi^2 (I\sigma)} = \Theta((I\sigma)^{-1}) = \Theta(N^{-\beta}\log^2 n).$$ Define constants
\begin{equation}
    \label{eqn: s and tau def}
s:= n\sigma^4 q_I^2 = \Theta\left(\frac{n N^{-2\beta}}{\log ^4N}\right)\qquad \hbox{ and } \qquad \tau_i := 1 - \frac{\delta \pi_i}{2\pi_I}.
\end{equation}

In particular note that since $\pi_I \geq \pi_i$ then $\tau_i \geq 1 - \frac{\delta}{2}$.

Define the following events:
\begin{equation}
    \label{eqn: event N^+ def}
    \mathcal{N}_i^+ := \left\{|N_{\Gamma_i}(vx) \cap A| \leq p |A| (1 + N^{\frac{1}{4} + \beta}) \text{ for all } A \subseteq [N], v \in [N], x \in [N]\right\}
\end{equation}
\begin{equation}
    \label{eqn: event Q^+}
\mathcal{Q}_i^+ = \left\{|O_i(A,B,x)| \leq q_i|A| |B| \text{ for all disjoint } A, B \in {[N]\choose \ge s}, x \in [N] \setminus A \cup B\right\}
\end{equation}
\begin{equation}
    \label{eqn: event Q_i def}
    \mathcal{Q}_i := \left\{\tau_i q_i |A||B| \leq |O_i(A,B,x)| \leq q_i |A||B| \text{ for all disjoint } A, B \in {[N]\choose n}, x \in [N] \setminus A \cup B\right\}.
\end{equation}

We will be interested in the intersections of these events and hence we define
\begin{equation}
    \label{eqn: event G_i}
    \notag
\mathcal{G}_i = \mathcal{N}_i \cap \mathcal{P}_i \cap \mathcal{P}^+ \cap \mathcal{Q}_i^+ \cap \mathcal{Q}_i \cap \mathcal{N}_i^+
\end{equation}
\begin{equation}
    \label{eqn: event G_<= i}
    \notag
    \mathcal{G}_{\leq i} = \bigcap_{j=0}^{i} \, \mathcal{G}_j.
    \end{equation}

Recall that $H=H_I$ and $\rho=\sqrt{\beta (\log N)/3N}$. Define the event 
\begin{equation}
    \label{eqn: Event T}
\mathcal{T}:= \{e_{H, x}(A,B) = (1 \pm \delta) \rho n^2 \text{ for all disjoint $n$-sets } A, B \subseteq V(H), x \in V(H) \setminus (A \cup B)\}
\end{equation}

We note that heuristically the number of edges we pick for $e_{H,x}(A,B)$ at each stage is $pq_i|A||B| = pq_in^2$ and
\begin{equation}
    \label{eqn: rho heuristic}
    \notag
\sum_{i=1}^I p q_i = \sum_{i=1}^I \frac{q_i \sigma}{\sqrt N} \approx \frac{\pi_I}{\sqrt{N}} \approx\sqrt{\frac{\beta \log N}{3N}} = \rho
\end{equation}
where we used the definition of $\pi_i$ from (\ref{eqn: pi def}) for the third relation and the fact that $\pi_i \approx \Psi(i\sigma) \approx \sqrt{(1/3)\log(\sqrt{3}i\sigma)}$ from Lemma \ref{q and pi bounds} and $I= \lceil N^{\beta} \rceil$ for the fourth relation.
Thus, ignoring the error parameters, we expect that if $\mathcal{Q}_i$ holds for all $i \le I$, then $\mathcal{T}$ holds as well.

The events $\mathcal{G}_i$ and $\mathcal{G}_{\leq i}$ are  good events for the algorithm at step $i$ and for all events up to and including step $i$ respectively. Our main result is that the events $\mathcal{T} \cap \mathcal{G}_{\leq I}$ hold with high probability, which then means the final $K_4^-$-free hypergraph $H_I$ at the end of the process will have all of the desired properties to prove Theorem \ref{main_theorem}.

\begin{lemma}
\label{main_lemma}
For $\delta, \beta, C$ as in the statement of Theorem \ref{main_theorem}, we have that $\mathbb{P}(\mathcal{T} \cap \mathcal{G}_{\leq I}) \geq 1 - N^{-\omega(1)}$.
\end{lemma}

In order to prove Lemma \ref{main_lemma}, we will prove that the probability we are not in the event $\mathcal{G}_{i+1}$ conditioned on being in the event $\mathcal{G}_i$ is exponentially small, and the probability of being in the event $\neg\mathcal{T} \cap \mathcal{G}_{\le I}$ is also exponentially small, such that when we multiply over the number of steps $I = N^{\beta}$ we still get a small probability. More formally, we will prove the following Lemma.

\begin{lemma}
\label{secondary_lemma}
Under the setup for Theorem \ref{main_theorem}, for all  $i \in [I-1]$,
$$\mathbb{P}(\neg \mathcal{G}_{\leq i+1} | \mathcal{G}_{\leq i}) \leq N^{-\omega(1)} \qquad \text{ and } \qquad \mathbb{P}(\neg \mathcal{T} \cap \mathcal{G}_{\leq I}) \leq N^{-\omega(1)}.$$
\end{lemma}

\section{Concentration bounds}

Throughout the paper we will use the following two standard bounds.
The following theorem is a well known version of the standard Chernoff bound.

\begin{theorem} [\cite{McDiarmid1998, Warnke}]
\label{standard chernoff}
    Let $(X_{\alpha})_{\alpha \in J}$ be a finite set of independent $(0,1)$ random variables and let $X = \sum_{\alpha \in J} X_{\alpha}$. Then if $\lambda = \mathbb{E}(X)$ $$\mathbb{P}(X \geq \lambda + t) \leq \exp \left(-\frac{t^2}{2\lambda}\right)$$
    $$\mathbb{P}(X \leq \lambda - t) \leq \exp \left(-\frac{t^2}{2(\lambda + t)} \right)$$
\end{theorem}

The next theorem of Mcdiarmid~\cite{McDiarmid1989, McDiarmid1998}, and Warnke~\cite{Warnke}  bounds a function of $0, 1$ random variables when the differences between two inputs of the function are bounded when the inputs only differ in one place.

\begin{theorem} [\cite{McDiarmid1989, McDiarmid1998, Warnke}]
    \label{chebyshev's inequality varraint}
    Let $(\xi_{\alpha})_{\alpha \in J}$ be a finite set of independent $0,1$ variables. Let $f: \{0,1\}^J \to \mathbb{R}$ be a decreasing function such that there exists $(c_{\alpha})_{\alpha \in J}$ so for all $z, z' \in \{0,1\}^J$ with $z_{\beta} = z'_{\beta}$ for all $\beta \neq \alpha$ we have $|f(z) - f(z')| \leq c_{\alpha}$. Set $\lambda := \sum_{\alpha \in J} c_{\alpha}^2 \mathbb{P}(\xi_{\alpha} = 1)$ and $X := f((\xi_{\alpha})_{\alpha \in J})$. For all $t \geq 0$,
        \begin{equation}
    \mathbb{P}(X \geq \mathbb{E}(X) + t) \leq \exp \left(-\frac{t^2}{2\lambda}\right).
    \end{equation}
    Further, if $C := \max_{\alpha \in J} c_{\alpha}$ and if we drop the decreasing assumption on $f$, then
  \begin{equation}
        \mathbb{P}(X \leq \mathbb{E}(X) - t) \leq \exp\left(-\frac{t^2}{2(\lambda + Ct)}\right)
    \end{equation}
\end{theorem}

The following Theorem which follows from Theorem 9 of Warnke~\cite{Warnke2017}, bounds the number of events which occur from a family where each has overlap with a limited number of events in the family.
\begin{theorem} [see~\cite{Warnke2017}]
    \label{Chernoff variant limited overlap}
    Let $\{\xi_i\}_{i \in O}$ be a family of independent variables with values  in $\{0,1\}$. Let $(Y_{\alpha})_{\alpha \in L}$ be a finite family of variables $Y_{\alpha} := \mathbbm{1}_{\xi_i = 1 \forall i \in \alpha}$ with $\sum_{\alpha \in L} \mathbb{E}(Y_{\alpha}) \leq \mu$. Define the function $f: 2^L \to \mathbb{R}$ by $f(J) = \max_{\beta \in J} |\alpha \in J: \alpha \cap \beta \neq \emptyset|$ for all $J \subseteq L$. Let $Z_C := \max \sum_{\alpha \in J} Y_{\alpha}$ with the maximum taken over $J \subseteq L$ with $f(J) \leq C$.
    Then for all $C, t \geq 0$, $$\mathbb{P}(Z_C \geq \mu + t) \leq \exp\left(-\frac{t^2}{2C(\mu + t)}\right)$$
\end{theorem}

Since our algorithm forms our final graph $H_I$ by removing a maximal set $D_{i+1}$ of bad sets of edges, we will use the following Theorem of Krivelevich~\cite{Krivelevich} to bound the number of edges removed.

\begin{theorem} [\cite{Krivelevich}]
\label{maximal subset bound}
    Let $Q$ be a finite set and let $(\xi_i)_{i \in Q}$ be a set of indicator random variables with $\mathbb{P}(\xi_i = 1) = p_i$ for all $i \in Q$. Let $(Q(\alpha))_{\alpha \in J}$ be a family of subsets of $Q$ where $J$ is a finite set. Let $X_{\alpha} = \prod_{i \in \alpha} \xi_{i}$ and let $X = \sum_{\alpha \in J} X_{\alpha}$. Let $$X_0 := \max \{m : \exists \alpha_1 ,..., \alpha_m \in J \text{ with } X_{\alpha} = 1 \text{ and } Q(\alpha_i) \cap Q(\alpha_j) = \emptyset, i \neq j\}.$$
    Then $$\mathbb{P}(X_0 \geq k) \leq \frac{\mathbb E(X)^k}{k!}.$$
\end{theorem}

\section{Proof of Lemma \ref{secondary_lemma}}
We will prove the statement about $\mathbb{P}(\neg \mathcal{G}_{\leq i+1} | \mathcal{G}_{\leq i})$ by proving each of the statements $$\mathbb{P}(\neg \mathcal{N}_{i+1} | \mathcal{G}_{\leq i}) \leq N^{-\omega(1)}$$
$$\mathbb{P}(\neg \mathcal{P}_{i+1} | \mathcal{G}_{\leq i}) \leq N^{-\omega(1)}$$
$$\mathbb{P}(\neg \mathcal{P}^+_{i+1} | \mathcal{G}_{\leq i}) \leq N^{-\omega(1)}$$
$$\mathbb{P}(\neg \mathcal{N}_{i+1}^+ | \mathcal{G}_{\leq i}) \leq N^{-\omega(1)}$$
$$\mathbb{P}(\neg \mathcal{Q}_{i+1}^+ | \mathcal{G}_{\leq i}) \leq N^{-\omega(1)}$$
$$\mathbb{P}(\neg \mathcal{Q}_{i+1} | \mathcal{G}_{\leq i}) \leq N^{-\omega(1)}$$
Before proceeding to the proofs, we prove two lemmas that will be used throughout the paper. In the lemmas below, we assume that $\mathcal{G}_{\leq i}$ holds.
\begin{lemma}
    \label{probability C^1 or Y}
    For all $e \in O_i$, we have that $\mathbb{P}(e \not\in C_{i+1}^1 \cup Y_{i+1}) - \frac{q_{i+1}}{q_i} \in [-7\sigma^{3/2}q_i, -5\sigma^{3/2}q_i]$.
\end{lemma}
\begin{proof}
Let $e = xuv$ and first notice that since we are in the event $\mathcal{G}_{\leq i}$ then $\mathcal{P}_i$ holds and $|S_{i}(x,uv)| \leq 2q_i\pi_i \sqrt{N}$. Therefore, $$|\hat{S}_i(e)| \leq |S_i(x,uv)| + |S_i(u,xv)| + |S_i(v,xu)| \leq 6q_i\pi_i \sqrt{N},$$ so $6\sqrt{N}q_i(\pi_i + \sqrt{\sigma}) - |\hat{S}_i(e)| \geq 0$ and $\hat{p}_{e,i} = 1 - (1-p)^{6\sqrt{N}q_i(\pi_i + \sqrt{\sigma}) - |\hat{S}_i(e)|}$. Now we get
\begin{equation}
\label{eqn: Probability not in C^1 or Y}
\notag
\mathbb{P}(e \not \in C_{i+1}^1 \cup Y_{i+1}) = (1-p)^{|\hat{S}_i(e)|}(1-\hat{p}_{e,i}) = (1-p)^{6\sqrt{N}q_i(\pi_i + \sqrt{\sigma})}.
\end{equation}
Then using (\ref{eqn: sigma and p definition}) and $1-ab \leq (1-a)^b \leq 1 - ab + a^2b^2$ when $a \in [0,1]$ and $b \geq 2$ we get
\begin{equation}
\label{eqn: Probability not in C^1 or Y Bound1}
1 - 6 \sigma q_i(\pi_i + \sqrt{\sigma}) \leq \mathbb{P}(e \not\in C_{i+1}^1 \cup Y_{i+1}) \leq 1 - 6 \sigma q_i(\pi_i + \sqrt{\sigma}) + 36 \sigma^2 q_i^2 (\pi_i + \sqrt{\sigma})^2.
\end{equation}
Now since $|(q_{i+1} - q_i) + 6\sigma q_i^2 \pi_i| \leq 16\sigma^2q_i^2$ by Lemma \ref{q and pi bounds} (\ref{bound on q_i - q_i+1}), we can rearrange to get $$\left|\frac{q_{i+1}}{q_i} - (1 - 6\sigma q_i \pi_i)\right| \leq 16\sigma^2q_i.$$ Then (\ref{bound on q_i - q_i+1}) together with (\ref{eqn: Probability not in C^1 or Y}) gives
\begin{equation}
\label{eqn: Probability not in C^1 or Y Bound 2}
\notag
\frac{q_{i+1}}{q_i} - 6\sigma^{3/2}q_i - 16\sigma^2 q_i \leq \mathbb{P}(e \not\in C_{i+1}^1 \cup Y_{i+1}) \leq \frac{q_{i+1}}{q_i} - 6\sigma^{3/2}q_i + 16\sigma^2 q_i + 36 \sigma^2 q_i^2 (\pi_i + \sqrt{\sigma})^2.
\end{equation}
Applying Lemma \ref{q and pi bounds} (\ref{q and pi bounds, q times pi less than 1}) and $\sigma = \log^{-2} N$, this completes the proof.
\end{proof}

The next lemma is notable and particularly useful as  it applies to every subset $J$ of triples.
\medskip

\begin{lemma}
\label{S_e intersection J}
    Given $\mathcal{G}_{\leq i}$, then for any $J \subseteq \binom{[N]}{3}$ we have that $\sum_{e \in O_i} |\hat{S}_i(e) \cap J| \leq 6q_i\pi_i\sqrt{N} |J|$.
\end{lemma}
\begin{proof}
    Note that for all $e \in O_i$, if $f \in \hat{S}_i(e)$ then there is a $K_4^-$ which contains $e$ and $f$ which are both in $O_i$ and some edge $g \in E_i$, which means this same $K_4^-$ satisfies the requirements for $e \in \hat{S}_i(f)$. Thus 
    \begin{equation}
    \notag
        \sum_{e \in O_i} |\hat{S}_i(e) \cap J| = \sum_{f \in J}\sum_{e \in O_i} \mathbbm{1}_{f \in \hat{S}_i(e)} \leq \sum_{f \in J} \sum_{e \in O_i} \mathbbm{1}_{e \in \hat{S}_i(f)} = \sum_{f \in J} |\hat{S}_i(f)| \leq |J|(2q_i)\pi_i \sqrt{N},
    \end{equation}
where the last inequality comes from the event $\mathcal{P}_i$.
\end{proof}
\subsection{Bound on $\neg \mathcal{N}$}
Throughout this section all expectations and probability are conditioned on $\mathcal{F}_i$ and on the event $\mathcal{G}_{\leq i}$, and to improve readability, we will omit that in our notation in this subsection. For each concentration bound, we fix $v$ and $x$ till the very end  when we take a union bound over all pairs $v, x$. 

Clearly $\mathbb{E}(|N_{\Gamma_{i+1}}(vx)|) = p |N_{O_i}(vx)| \leq pq_iN =: \mu$ since we are in the event $\mathcal{G}_{\leq i}$, so $\mathcal{N}_i$ holds. Now by Theorem \ref{standard chernoff}
\begin{equation}
\label{eqn: Prob Gamma neighborhood bound}
\mathbb{P}(|N_{\Gamma_{i+1}}(vx)| \geq 2\mu) \leq \exp (\frac{-\mu^2}{2\mu}) \leq \exp(-\frac{\sigma \sqrt{N}q_i}{2}).
\end{equation}
Then by Lemma \ref{q and pi bounds} (\ref{eqn:sigmaqN}) and (\ref{eqn: Prob Gamma neighborhood bound}) we obtain $\mathbb{P}(|N_{\Gamma_{i+1}}(vx)| \geq 2\mu) \leq e^{-\frac{N^c}{3}} < N^{-\omega(1)}$ where the $c > 0$ comes from (\ref{q_i sqrt N estimate}). Since the number of pairs $v,x$ is $\binom{N}{2}$,  a union bound completes the proof.

Next we consider $|N_{O_i}(vx)|$. First we will give an upper bound on $\mathbb{E}(|N_{O_i}(vx)|)$. Recall from (\ref{eqn: O_i change def}) that $$O_i \setminus O_{i+1} = \Gamma_{i+1} \cup C^1_{i+1} \cup C^2_{i+1} \cup Y_{i+1}.$$ Hence $O_{i+1} \subseteq O_i \backslash (C_{i+1}^1 \cup Y_{i+1})$, so 
\begin{equation}
\label{N_i(vx) bound definition of X}
\mathbb{E}(|N_{O_{i+1}}(vx)|) \leq \sum_{u \in N_{O_{i}}(vx)} \mathbb{P}(xvu \not\in C_{i+1}^1 \cup Y_{i+1}) =: X.
\end{equation}
Recall that $X= \sum_{u \in N_{O_{i}}(vx)} \mathbb{P}(xvu \not\in C_{i+1}^1 \cup Y_{i+1})$. Since we are in $\mathcal{G}_{\leq i}$,  $|N_{O_i}(vx)| \le q_iN$ by $\mathcal{N}_i$, and therefore by Lemma~\ref{probability C^1 or Y}
\begin{equation} 
\label{eqn: N_O(vx) Expected Value Upper Bound}
\mathbb{E}(X) \leq |N_{O_i}(vx)| \cdot \mathbb{P}(xvu \not\in C_{i+1}^1 \cup Y_{i+1}) \leq q_iN\left(\frac{q_{i+1}}{q_i} -\sigma^{3/2} q_i\right) = Nq_{i+1} - \sigma^{3/2}q_{i}^2N.
\end{equation}
 Next, we will prove concentration around the expected value by using Theorem \ref{chebyshev's inequality varraint}. The index set $J$ from Theorem \ref{chebyshev's inequality varraint} will be made of two parts, $J = J_{\Gamma} \cup J_Y$ where both $J_{\Gamma}$ and $J_Y$ are copies of $O_i$. Then $\xi_{\alpha}$ from Theorem \ref{chebyshev's inequality varraint} will be defined as follows:
\begin{equation}
\label{xi defintion}
\xi_{\alpha} = \begin{cases} 
\mathbbm{1}_{\alpha \in \Gamma_{i+1}} \qquad \hbox{if } \alpha \in J_{\Gamma},\\
\mathbbm{1}_{\alpha \in Y_{i+1}} \qquad \hbox{if } \alpha \in J_{Y}.
\end{cases}
\end{equation}
  Then $f$ in Theorem \ref{chebyshev's inequality varraint} will be defined so that $X$ from Theorem \ref{chebyshev's inequality varraint} is  the same as $X$ from (\ref{N_i(vx) bound definition of X}), and notice that $f$ is decreasing since the presence of edges in $\Gamma_{i+1}$ or $Y_{i+1}$ only adds edges to $C^1_{i+1} \cup Y_{i_+1}$. We need to give bounds on $\Delta_e$ which is the absolute value of the change in the size of
$(C_{i+1}^1 \cup Y_{i+1}) \cap (N_{O_i}(vx) \times \{v,x\})$ when we change whether or not an edge $e$ is in $\Gamma_{i+1}$. Assume for now that  $\Delta_e \le c_e$. Similarly, let $\Upsilon_e$ be the absolute value of the change in $|(C_{i+1}^1 \cup Y_{i+1}) \cap (N_{O_i}(vx) \times \{v,x\})|$ when we change whether or not an edge $e$ is in $Y_{i+1}$ and assume that $\Upsilon_e \le c'_{e}$.

Observe that changing whether or not $e$ is in $\Gamma_{i+1}$ only affects which edges are in $C_{i+1}^1$, and similarly changing whether an edge is in $Y_{i+1}$ will only affect which edges are in $Y_{i+1}$. Consequently, when considering $c_e$ we only need to consider the effect on $C_{i+1}^1$, and when considering $c'_e$ we only need to consider the effect on $Y_{i+1}$. 

First, to give a bound on $\Delta_e$, notice that $\Delta_e$ by definition is the number of changes to
$$C_{i+1}^1  \cap (N_{O_i}(vx) \times \{v,x\})=
\{xvw \in O_i : \hat{S}_{i}(xvw) \cap \Gamma_{i+1} \neq \emptyset\}.$$ 
Since we only changed whether or not $e$ was in $\Gamma_{i+1}$, $e$ can only affect elements of $C_{i+1}^1  \cap (N_{O_i}(vx) \times \{v,x\})$
when $e \in \hat{S}_{i}(vxw)$. Further, if $e \in \hat{S}_{i}(vxw)$ then $vxw \in \hat{S}_i(e)$. Indeed,  $e \in \hat{S}_{i}(vxw)$ implies $e \in O_i$ and that there exists  $e' \in E_i$ with $\{vxw\}ee'$ forming a $K_4^-$. Then $e \in O_i$ and $e' \in E_i$ so $e\{vxw\}e'$ is a $K_4^-$ which is the definition for $vxw \in \hat{S}_i(e)$.  So the number of $w$ such that $vxw$ is in 
$C_{i+1}^1  \cap (N_{O_i}(vx) \times \{v,x\})$
is at most $|\hat{S}_i(e) \cap (\{vx\} \times N_{O_i}(vx))|$. Thus for $e=u_1u_2u_3$, 
\begin{equation} \label{eqn:Debound}
\Delta_e \leq |\hat{S}_i(e) \cap ( N_{O_i}(vx) \times \{v,x\})| \leq |\hat{S}_i(e)| \leq 6q_i\pi_i \sqrt{N}\end{equation}
due to the bound $\max\{\hat{S}_{i}(u_1,u_2u_3), S_i(u_2,u_1u_3), S_i(u_3,u_1u_2)\} \le 2q_i\pi_i \sqrt{N}$
as we are assuming the
event $\mathcal{P}_i$. Further, $\sum_{e \in O_i} |\hat{S}_i(e) \cap (N_{O_i}(vx) \times \{v,x\})|\le 6q_i \pi_i \sqrt{N} |N_{O_i}(vx)|$ by Lemma~\ref{S_e intersection J}.
Moreover, by (\ref{eqn:Debound}) and Lemma~\ref{S_e intersection J},
\begin{equation}
\label{eqn: O_i delta e term}
\sum_{e \in O_i} \Delta_e^2 \le 6q_i\pi_i \sqrt{N}\sum_{e \in O_i} \Delta_e
\leq 6q_i\pi_i \sqrt{N} \sum_{e \in O_i} |\hat{S}_{i}(e) \cap (N_{O_i}(vx) \times \{v,x\})| \leq 36 q_i^2 \pi_i^2 N (q_i N)
\end{equation}
where the bound on $|N_{O_i}(vx) \times \{v,x\}|$ in the last inequality comes from the event $\mathcal{N}_i$.

Next, we want to bound $\Upsilon_e$, but notice that since all edges are in $Y_{i+1}$ independently, changing whether or not $e$ is in $Y_{i+1}$ can only affect  whether the edge $e$ itself (and no other edge) is in $(C_{i+1}^1 \cup Y_{i+1}) \cap (N_{O_i}(vx) \times \{v,x\})$. Therefore $\Upsilon_e \le 1$ and if $\Upsilon_e = 1$, then $e\in N_{O_i}(vx) \times \{v,x\}$.
Thus by $\mathcal{N}_i$, 
\begin{equation}
\label{eqn: O_i Upsilon term}
\sum_{e \in O_i} \Upsilon_e^2 \leq |N_{O_i}(vx) \times \{v,x\}| = |N_{O_i}(vx)| \leq q_i N.
\end{equation}
Now we can bound the value of $\lambda$ from Theorem \ref{chebyshev's inequality varraint} using the fact that $\hat{p}_{e,i} \leq q_i$ from Lemma \ref{q and pi bounds} (\ref{p hat < q}) and the bounds (\ref{eqn: O_i delta e term}) and (\ref{eqn: O_i Upsilon term}) by

\begin{equation}
\notag
\lambda := \sum_{e \in O_i} p\Delta_e^2 + \sum_{e \in O_i} \hat{p}_{e,i}\Upsilon_e^2 \leq p(36q_i^3\pi_i^2 N^2) + q_i^2N \leq 37 \sigma q_i^2 N^{3/2}
\end{equation}
where the final inequality followed from Lemma \ref{q and pi bounds} (\ref{q and pi bounds, q times pi less than 1}). Finally, we apply Theorem \ref{chebyshev's inequality varraint} with $\lambda$ and (\ref{eqn: N_O(vx) Expected Value Upper Bound}) to obtain $$\mathbb{P}(|N_{O_{i+1}(vx)}| \geq q_{i+1}N) \leq \exp \left(-\frac{\sigma^3 q_i^4 N^2}{74\sigma q_i^2 N^{3/2}}\right) = \exp \left(-\left(\frac{1}{74}\right)\sigma^2 q_i^2 \sqrt{N}\right) \leq N^{-\omega(1)}$$ where the final inequality holds from Lemma \ref{q and pi bounds} (\ref{eqn:sigmaqN}). Now taking a union bound over the $\binom{N}{2}$ choices for $v$ and $x$ completes the proof.

\subsection{Bound on $\neg \mathcal{P}$}

Throughout this section we will again omit the conditioning on $\mathcal{F}_i$ and the event $\mathcal{G}_{\leq i}$ in our notation.

We will start with the proof for $T_{i}(x,uv)$. Clearly
\begin{equation}
\label{eqn: Tchange}
|T_{i+1}(x,uv)| - |T_{i}(x,uv)| = \sum_{w \in R_{i}(x,uv)} \mathbbm{1}_{\{xuw, xvw \in \Gamma_{i+1}\}} + \sum_{w \in S_{i}(x,uv)} \mathbbm{1}_{\{xuw \in \Gamma_{i+1} \text{ or } xvw \in \Gamma_{i+1}\}}
\end{equation}
Notice that in the second term of (\ref{eqn: Tchange}) for every $w$ in $S_{i}(x,uv)$ one of $xuw$ or $xvw$ is already in $E_i$ so only one can potentially be in $\Gamma_{i+1}$. Now since for edges in $O_i$ the event of being in $\Gamma_{i+1}$ are independent Bernoulli random variables with probability $p$, we can bound the expected change in $T_{i}(x,uv)$ using the bounds on $|R_{i}(x,uv)|$ and $|S_{i}(x,uv)|$ from  (\ref{eqn: event P_i}) as follows.
\begin{equation}
    \label{eqn:expectedTchange}
    \notag
    \mathbb{E}(|T_{i+1}(x,uv)| - |T_{i}(x,uv)|) \leq (p^2) q_i^2 N + (p)2q_i \pi_i \sqrt{N} = \sigma^2 q_i^2 + 2q_i \pi_i \sigma \leq \sigma^2 + 2\sqrt{\sigma} \ll 1
\end{equation}
where in the last two inequalities we used $\sqrt{\sigma}\pi_i \leq 1$ and $q_i \leq 1$ and $\sigma \ll 1$ from Lemma \ref{q and pi bounds} (\ref{q less than 1}) and (\ref{sqrt sigma pi < 1}).

Now using the bound on $T_{i}(x,uv)$ from (\ref{eqn: event P_i}), $$|T_{i+1}(x,uv)| = |\Delta T_{i+1}(x,uv)| + |T_{i}(x,uv)| \leq |\Delta T_{i+1}(x,uv)| + i\log^9 N$$
Thus Theorem $\ref{standard chernoff}$ gives
\begin{align*}
\mathbb{P}(|T_{i+1}(x,uv)| \geq (i+1) \log^9 N) &\leq \mathbb{P}(|\Delta T_{i}(x,uv)| \geq \log^9 N) \\
&\leq \exp \left(-\frac{(\log^9N - \sigma^2q_i^2 - 2q_i\pi_i\sigma)^2}{2(\sigma^2q_i^2 + 2q_i\pi_i\sigma)}\right) \\
&\leq \exp \left(-\log^{17} N \right)\\
&\leq N^{-\omega(1)}
\end{align*}
where the third inequality comes from (\ref{q and pi bounds, q times pi less than 1}).
Thus taking a union bound over all $3 \binom{N}{3}$ choices for the vertex $x$ and the vertices $uv$ completes the proof.

Next, we will prove the bound on $|R_{i}(x,uv)|$. Since $O_{i+1} \subseteq O_i \backslash (C_{i+1}^1 \cup Y_{i+1})$ we get 
\begin{equation}
    \label{eqn: R_x,uv Bound 1}
|R_{i+1}(x,uv)| \leq \sum_{w \in R_{i}(x,uv)} \mathbbm{1}_{\{xuw,xvw \not \in (C_{i+1}^1 \cup Y_{i+1})\}} =: X
\end{equation}
Observe that
\begin{align*}
\mathbb{P}(xuw, xvw \not\in C_{i+1}^1 \cup Y_{i+1}) &\leq (1 - \hat{p}_{xuw,i})(1 - \hat{p}_{xvw,i})(1-p)^{|\hat{S}_{i}(x,uw)| + |\hat{S}_{i}(x,vw)| - |\hat{S}_{i}(x,uw) \cap \hat{S}_{i}(x,vw)|} \\
&= \mathbb{P}(xuw \not\in C_{i+1}^1 \cup Y_{i+1})\mathbb{P}(xvw \not \in C_{i+1}^1 \cup Y_{i+1})(1-p)^{-|\hat{S}_{i}(x,uw) \cap \hat{S}_{i}(x,vw)|}
\end{align*}
where the first inequality comes from excluding some edges in $\hat{S}_i(u,xw), \hat{S}_i(w,ux), 
\hat{S}_i(v,ux),$ and $\hat{S}_i(w,vx)$ which could also cloase edges.

Now if $e \in |\hat{S}_{i}(x,uw) \cap \hat{S}_{i}(x,vw)|$ then since $u \neq v$, $e$ must be of the form $xwz$ for some $z \in [N] \setminus \{x,u,v,w\}$, and $xwz \in O_i$ and $xuz, xvz \in E_i$. This implies $z \in T_{i}(x,uv)$, so since $\mathcal{P}_i$ holds the number of $z$ satisfying these conditions is bounded above by $i\log^9 N$ by (\ref{eqn: event P_i}). Thus 
\begin{equation}
\label{eqn: Intersection S_x,uw S_x,vw Bound}
|\hat{S}_{i}(x,uw) \cap \hat{S}_{i}(x,vw)| \leq i\log^9 N.
\end{equation}
Combining (\ref{eqn: R_x,uv Bound 1}), (\ref{eqn: Intersection S_x,uw S_x,vw Bound}), (\ref{eqn: event P_i}), and Lemma \ref{probability C^1 or Y}, we bound the expected value of $X$ by
\begin{align*}
    \mathbb{E}(X) &\leq (q_i^2 N)(\frac{q_{i+1}}{q_i} - 5\sigma^{3/2}q_i)^2 (1-p)^{-i \log^9 N} \\
    &\leq \left(q_{i+1}^2 N - 5q_{i+1}q_i^2 \sigma^{3/2} N + 25q_i^4 \sigma^3 N \right) \left(1 + pI \log^9 N \right) \\
    &\leq q_{i+1}^2 N -5q_i^3(1-12\sigma) \sigma^{3/2} N + 25q_i^4 \sigma^3 N + q_i^2 Np I \log^9 N  + 25q_i^4 \sigma^3 N p I \log^9 N\\
    &\leq q_{i+1}^2 N - 4.9q_i^3\sigma^{3/2}N + 25q_i^4 \sigma^3 N + \sqrt{N}\left(q_i^2 \sigma I \log^9 N + 25q_i^4 \sigma^4 I \log^9 N \right) \\
    &\leq q_{i+1}^2N - 4q_i^3 \sigma^{3/2} N
\end{align*}
where the third inequality uses $q_{i+1} \geq q_i(1-12\sigma)$ from Lemma \ref{q and pi bounds} (\ref{q_i - q_i+1 < sigma min...}) and $q_{i+1} \leq q_i$, and the last inequality uses $q_i^3\sigma^{3/2} \gg q_i^4 \sigma^3$ and $q_i^3\sigma^{3/2} N \gg \sqrt{N}q_i^2 \sigma I \log^9 N$ by Lemma \ref{q and pi bounds} (\ref{eqn:sigmaqN}).

To prove concentration, we will apply Theorem \ref{chebyshev's inequality varraint}. We define the index set $J= J_{\Gamma} \cup J_Y$ and define $\xi_{\alpha}$ as in (\ref{xi defintion}). Then $f$ will be defined so that $X = f((\xi_{\alpha})_{\alpha \in J})$. We again let $\Delta_e$ be the absolute value change of $X$ when we change whether an edge $e$ is in $\Gamma_{i+1}$, and let $\Upsilon_e$ be the absolute value change of $X$ when we change whether an edge $e$ is in $Y_{i+1}$. Then notice that if changing whether $e \in \Gamma_{i+1}$ changes whether $xuw$ or $xvw$ is in $C^1_{i+1}$, for some $w \in R_i(x,uv)$, there exists $e' \in E_i$ so $ee'\{xuw\}$ or $ee'\{xvw\}$ is a $K_4^-$ and then either $\{xuw\}$ or $\{xvw\}$ is in $\hat{S}_i(e)$. Thus $\Delta_e \leq |\hat{S}_i(e) \cap \left(\{xu, xv\} \times R_{i}(x,uv)\right)|$ and $\Upsilon_e$ is only $1$ when $e = xuw$ or $e = xvw$ for $w \in R_{i}(x,uv)$ and $0$ otherwise since changes edges being in $Y_{i+1}$ doesn't impact other edges. Then we can bound $\lambda$ in Theorem \ref{chebyshev's inequality varraint} by
\begin{align*}
    \lambda &= \sum_{e \in O_i} p \Delta_e^2 + \sum_{e \in O_i} \hat{p}_{e,i} \Upsilon_e^2\\
    &\leq p(6q_i\pi_i \sqrt{N})\sum_{e \in O_i} |\hat{S}_i(e) \cap (\{xu,xv\} \times R_i(x,uv))| + q_i(2q_i^2 N) \\
    &\leq \frac{\sigma}{\sqrt{N}} (6q_i\pi_i \sqrt{N})^2 (2q_i^2 N) + 2q_i^3 N\\
    &\leq 100\sigma q_i^3 N^{3/2}
\end{align*}
where the second inequality uses Lemma \ref{S_e intersection J} together with the fact that $\{xu, xv\} \times R_{i}(x,uv) \leq 2|R_{i}(x,uv)|$ and the last inequality uses Lemma \ref{q and pi bounds} (\ref{q and pi bounds, q times pi less than 1}).
Now Theorem~\ref{chebyshev's inequality varraint} gives
\begin{align}
    \mathbb{P}(|R_{i+1}(x,uv)| \geq q_{i+1}^2 N) 
    &\leq \mathbb{P}(X \geq \mathbb{E}(X) + 4q_{i}^3 \sigma^{3/2} N) \notag \\
    &\leq \exp\left(-\frac{16q_i^6 \sigma^3 N^2}{200\sigma q_i^3 N^{3/2}}\right) \notag  \\ &= \exp \left(-\frac{4}{25}q_i^3 \sigma^2 \sqrt{N})\right) \leq N^{-\omega(1)} \notag
\end{align}
where the last inequality follows from Lemma \ref{q and pi bounds}. Taking a union bound over the $3\binom{N}{3}$ choices for $x$ and $uv$ completes the proof for all of the $R_{i}(x,uv)$ variables.

Lastly we will prove the bounds for the variables $S_{i+1}(x,uv)$. Notice there are two types of edges which can be included in $\hat{S}_{i+1}(x,uv)$.
The first type are edges which were in $\hat{S}_{i}(x,uv)$ and did not leave $\hat{S}_{i}(x,uv)$ during step $i+1$.
The second type are edges $xuw$ where $w \in R_{i}(x,uv)$ and $xvw \in \Gamma_{i+1}$ or edges $xvw$ where $w\in R_i(x,uv)$ and $xuw \in \Gamma_{i+1}$.

We will prove bounds on the number of edges in $\hat{S}_{i}(x,uv)$ of each of these types separately and then combine the bounds to prove that the required upper bound for $S_{i}(x,uv)$ holds whp. To this end, define random variables
\begin{equation}
    \label{eqnn: S^1_x,uv def}
    \notag
    S_{i}^1(x,uv) := \sum_{e \in \hat{S}_{i}(x,uv)} \mathbbm{1}_{e \not \in C^1_{i+1} \cup Y_{i+1}}
\end{equation}

\begin{equation}
    \label{eqnn: S^2_x,uv def}
    \notag
    S_{i}^2(x,uv) := \sum_{w \in R_{i}(x,uv)} \mathbbm{1}_{xuw \in \Gamma_{i+1} \text{ or } xvw \in \Gamma_{i+1}}
\end{equation}

\begin{equation}
    \label{eqn: S bound by S^1 and S^2}
    |S_{i+1}(x,uv)| \leq S_{i+1}^1(x,uv) + S_{i+1}^2(x,uv)
\end{equation}
where inequality in (\ref{eqn: S bound by S^1 and S^2}) comes from the fact that edges can leave $\hat{S}_{i}(x,uv)$ if they are in the sets $\Gamma_{i+1}$ and $C_{i+1}^2$.
Because we are in the event $\mathcal{P}_i$,  (\ref{eqn: event P_i}) implies $|S_{i}(x,uv)|\le 2q_i\pi_i\sqrt N$ and $|R_{i}(x,uv)| \le q_i^2 N$.
Together with Lemma \ref{probability C^1 or Y} we obtain
\begin{align*}
    \label{eqnn: S^1 + S^2 expected value}
    \mathbb{E}(S_{i+1}^1(x,uv) + S_{i+1}^2(x,uv)) &\leq |S_{i}(x,uv)|\left(\frac{q_{i+1}}{q_i} - 5\sigma^{3/2}q_i\right) + 2p|R_{i}(x,uv)| \\ 
    &\leq 2q_{i+1}\pi_i \sqrt{N} - 10 q_i^2 \pi_i \sigma^{3/2} \sqrt{N} + 2\sigma \sqrt{N}q_i^2 \\
    &\leq 2q_{i+1}\pi_i\sqrt{N} - 10q_i^2\pi_i \sigma^{3/2}\sqrt{N} + 2\sigma \sqrt{N} q_i (q_{i+1} + 12\sigma q_i)\\
    &= 2q_{i+1} \sqrt{N} (\pi_i + \sigma q_i) - 10q_i^2 \pi_i \sigma^{3/2} \sqrt{N} + 24\sigma^2 q_i^2 \sqrt{N}\\
    &\leq 2q_{i+1}\pi_{i+1}\sqrt{N} -9q_i^2 \pi_i \sigma^{3/2}\sqrt{N}
\end{align*}
where the third inequality follows from (\ref{q_i - q_i+1 < sigma min...}) and the last inequality follows from the fact that $\pi_{i+1} = \pi_i + \sigma q_{i+1}$ by (\ref{eqn: pi def}) and $\sigma^2 \ll \pi_i \sigma^{3/2}$ by Lemma \ref{q and pi bounds} (\ref{Psi bound}) and (\ref{q and pi bounds pi close to psi}). Now if we can prove that whp the values of $S_{i+1}^1(x,uv)$ and $S_{i+1}^2(x,uv)$ are both not more than $q_i^2\pi_i\sigma^2 \sqrt{N}$ larger than their respective expected values, we will have proven the bound on $|S_{i}(x,uv)|$.

Now to prove concentration for $S_{i}^1(x,uv)$, we will bound $\lambda$ in Theorem \ref{chebyshev's inequality varraint}. Let $\Delta_e$ be the absolute value of the change in $S_{i}^1(x,uv)$ when we change whether $e$ is in $\Gamma_{i+1}$ and let $\Upsilon_e$ be the absolute value change in $S_{i}^1(x,uv)$ when we change whether $e$ is in $Y_{i+1}$.
Similar to the proofs of concentration for $|N_{O_i}(vx)|$ and $|R_{i}(x,uv)|$, $e$ only affects whether an edge $f$ is in $S_{i+1}(x,uv)$ or not if $e$ and $f$ are in a $K_4^-$ together with one other edge in $E_i$, which would then mean $f \in \hat{S}_i(e)$. Thus if we let $e = w_1w_2w_3$ $$\Delta_e \leq |\hat{S}_i(e) \cap \hat{S}_{i}(x,uv)| \leq |\hat{S}_i(e)| \leq |S_{i}(w_1,w_2w_3)| + |S_{i}(w_2,w_1w_3)| + |S_{i}(w_3,w_1w_2)|\leq 6q_i\pi_i\sqrt{N}.$$
However, we can obtain a much better bound on $|\hat{S}_i(e) \cap \hat{S}_{i}(x,uv)|$ than $|\hat{S}_i(e)|$ when $e\ne xuv$.
If $w_1w_2w_3 \cap xuv = \emptyset$ then $|\hat{S}_i(e) \cap \hat{S}_{i}(x,uv)| = 0$ since any triple which is in a $K_4^-$ with $e$ must contain two of the vertices $w_1,w_2, w_3$ but then this triple cannot contain both $x$ and one of $v$ or $u$.
If $|w_1w_2w_3 \cap xuv| = 1$, then $|\hat{S}_i(e) \cap \hat{S}_{i}(x,uv)| \leq 4$ since if a triple contains both two of $w_1,w_2,w_3$ and two of $x,u,v$, it must contain the unique overlapping vertex and then one of the remaining vertices from each of $w_1w_2w_3$ and $xuv$.
Finally, if $|w_1w_2w_3 \cap xuv| = 2$ then if $f \in \hat{S}_i(e) \cap \hat{S}_{i}(x,uv)$ it either includes one of the shared vertices between $w_1w_2w_3$ and $xuv$ and both of the verties in $w_1w_2w_3$ and $xuv$ which are not shared, or it contains both of the shared vertices and a third vertex.
Thus there are only two choices for $f$ which contain only one of the shared vertices. When $f$ contains both of the shared vertices, let the last vertex of $f$ be $z$. Notice that $f$ must contain $x$ by the definition of $\hat{S}_{i}(x,uv)$ and similarly $f$ must contain exactly one of $u,v$. 

{\bf Case 1.} Suppose that  $f = xuz$. Without loss of generality let $w_1 = x$ and $w_2 = u$ and $f = w_1w_2z$. Recall that $e=w_1w_2w_3$. Now since $f \in \hat{S}_i(e)$ there must be an edge in $E_i$ out of edges comprised of the vertices $x,u,w_3,z$ and $xuz,xuw_3 \in O_i$. Thus we have the following two subcases.

{\bf Subcase 1a.} $xw_3z \in E_i$. Since $xvz, xw_3z \in E_i$ we have  $z \in T_{i}(x,vw_3)$.

{\bf Subcase 1b.} $uw_3z \in E_i$. Since $xuz, xuv, xuw_3 \in O_i$ and $xvz, uw_3z \in E_i$ then $z \in U_{i}(x,u,v,w_3)$ (See Figure \ref{S^1 Cases} below).

{\bf Case 2.} Suppose that  $f = xvz$. As in Case 1, we obtain $z \in T_i(x, uw_3)$ or $z \in U_i(x,v,u,w_3)$.

\begin{figure}[hbt!]
\centering
\begin{tikzpicture}
\node (x) at (0,0) {};
\fill(x) circle(0.1) node[above]{$x$};
\node (v) at (-2,2) {};
\fill(v) circle(0.1) node[above]{$v$};
\node(u) at (2,2) {};
\fill(u) circle(0.1) node[above]{$u$};
\node (z) at (-2,-2){};
\fill(z) circle(0.1) node[below]{$z$};
\node (w) at (2,-2) {};
\fill(w) circle(0.1) node[below]{$w_3$};

\draw [rounded corners, dashed] (-2.6,2.4) -- (2.6,2.4) -- (0,-0.8) -- cycle;

\draw [rounded corners,dashed] (2.4,2.7) -- (-0.4,0) -- (2.4,-2.7) -- cycle;

\draw [rounded corners] (-2.4,2.7) -- (0.4,0) -- (-2.4,-2.7) -- cycle;

\draw[rounded corners] (-2.6,-2.4) -- (2.6,-2.4) -- (0,0.8) -- cycle;

\draw[rounded corners, dashed] (-3,-2) -- (-2,-3) -- (3,2) -- (2,3) -- cycle;

\node (Case 1) at (0,-3) {Case 1a};

\node (x) at (0+7,0) {};
\fill(x) circle(0.1) node[above]{$x$};
\node (v) at (-2+7,2) {};
\fill(v) circle(0.1) node[above]{$v$};
\node(u) at (2+7,2) {};
\fill(u) circle(0.1) node[above]{$u$};
\node (z) at (-2+7,-2){};
\fill(z) circle(0.1) node[below]{$z$};
\node (w copy) at (2+7,-2) {};
\fill(w copy) circle(0.1) node[below]{$w_3$};

\draw [rounded corners,dashed] (-2.6+7,2.4) -- (2.6+7,2.4) -- (0+7,-0.8) -- cycle;

\draw [rounded corners,dashed] (2.4+7,2.7) -- (-0.4+7,0) -- (2.4+7,-2.7) -- cycle;

\draw [rounded corners] (-2.4+7,2.7) -- (0.4+7,0) -- (-2.4+7,-2.7) -- cycle;

\draw[rounded corners] (-2.5+7,-1.7) -- (1.8 + 7, -1.7) -- (1.8 + 7, 2.7) -- (2.5 + 7, 2.7) -- (2.5 + 7, -2.5) -- (-2.5 + 7, -2.5) -- cycle;

\draw[rounded corners, dashed] (-3+7,-2) -- (-2+7,-3) -- (3+7,2) -- (2+7,3) -- cycle;

\node(Case 2) at (0+7, -3) {Case 1b};

\node(Open Key) at (-1,3.5) {Edges in $O_i$};
\node(Closed Key) at (6,3.5) {Edges in $E_i$};
\draw[dashed] (0.5,3.5) -- (2.5,3.5);
\draw (7.5,3.5) -- (9.5,3.5);

\end{tikzpicture}
\caption{}
\label{S^1 Cases}
\end{figure}

This means the number of $f \in \hat{S}_i(e) \cap \hat{S}_{i} (x,uv)$ is less than $4i(\log N)^9$.
Putting this together we get that if $e \neq xuv$ then $\Delta_e \leq \max \{4i (\log N)^9, 4\} \leq I\sigma^{-5}$. Hence applying Lemma \ref{S_e intersection J} with $J = \hat{S}_{i}(x,uv)$,
\begin{equation}
    \label{eqnn: S^1 lambda bound Delta term}
    \notag
    p \sum_{e \in O_i} \Delta_{e}^2 \leq p ((6q_i\pi_i \sqrt{N})^2 + I\sigma^{-5}(6q_i \pi_i \sqrt{N}) |S_{i}(x,uv)|) \leq 13I\sigma^{-4}\sqrt{N}q_i^2\pi_i^2
\end{equation}

Further, $\Upsilon_e \leq 1$ so $\hat{p}_{e,i} \sum_{e \in O_i} \Upsilon_e^2 \le q_i \sum_{e \in O_i} \Upsilon_e^2 \leq q_i |S_{x,uv}(i)| \leq 2q_i^2\pi_i \sqrt{N}$. Then we can bound $\lambda$ in Theorem \ref{chebyshev's inequality varraint} by
\begin{equation}
    \label{eqnn: lambda for S^1}
    \notag
    \lambda \leq p\sum_{e \in O_i} \Delta_e^2 + q_i\sum_{e \in O_i} \Upsilon_e^2 \leq 15I\sigma^{-4}\sqrt{N}q_i^2\pi_i^2.
\end{equation}
Theorem \ref{chebyshev's inequality varraint} gives
\begin{equation}
    \label{eqnn: S^1_x,uv bound}
    \mathbb{P}(S_{i}^1(x,uv) \geq \mathbb{E}(S_{i}^1(x,uv) + q_i^2\pi_i\sigma^{2}\sqrt{N}) \leq \exp(\frac{-q_i^4\pi_i^2\sigma^4 N}{30I\sigma^{-4}\sqrt{N}q_i^2\pi_i^2}) 
    = \exp(\frac{-q_i^2\sigma^8 \sqrt{N}}{30I})\leq N^{-\omega(1)}
\end{equation}
where the last inequality comes from (\ref{eqn:sigmaqN}).

Next, to prove concentration for $S_{i}^2(x,uv)$, notice that for all $w \in R_{i}(x,uv)$ the inclusion of the edges $xuw$ or $xvw$ in $\Gamma_{i+1}$ are independent identically distributed Bernoulli random variables. Also, $\mathbb{E}(S_{i}^2(x,uv)) \leq 2pq_i^2N=2\sigma q_i^2 \sqrt{N}$ and $\sigma \pi_i \leq 1$ by Lemma \ref{q and pi bounds} (\ref{sqrt sigma pi < 1}). We use Theorem \ref{standard chernoff} to obtain

\begin{equation}
    \label{eqnn: bound for S^2_x,uv}
    \mathbb{P}(S_{i}^2(x,uv) \leq \mathbb{E}(S_{i}^2(x,uv)) + \sigma^2 q_i^2 \pi_i \sqrt{N}) \leq \exp \left(\frac{-(\sigma^2 q_i^2 \pi_i \sqrt{N})^2}{2(2\sigma q_i^2 \sqrt{N})}\right) =\exp\left(\frac{-\sigma^3 q_i^2 \pi_i^2\sqrt{N}}{4}\right) \leq N^{-\omega(1)}
\end{equation}
where the last inequality holds by Lemma \ref{q and pi bounds}. Now taking a union bound over all choices for $x$ and $uv$ and combining (\ref{eqnn: S^1_x,uv bound}) and (\ref{eqnn: bound for S^2_x,uv}) completes the proof that $\mathbb{P}(\neg \mathcal{P}_{i+1} | \mathcal{G}_{\leq i}) \leq n^{-\omega(1)}$.

\subsection{Bound on $\neg \mathcal{P}^+$}

We will omit the conditioning on $\mathcal{F}_i$ and the event $\mathcal{G}_{\leq i}$ in the notation in this section. For $z$ to be in $U_{i+1}(x,u,v,w)$ but not in $U_i(x,u,v,w)$, at least one of the edges $xvz, uzw$ must have been added to $\Gamma_{i+1}$ in this round (see Figure \ref{U fig}). In other words there are three cases for how $xvz$ and $uwz$ are in $E_{i+1}$ 
\begin{enumerate}
    \item $xvz \in \Gamma_{i+1}$ and $uwz \in E_{i}$
    \item $xvz \in E_{i}$ and $uwz \in \Gamma_{i+1}$
    \item $xvz \in \Gamma_{i+1}$ and $uwz \in \Gamma_{i+1}$.
    \end{enumerate}
Further, for $z$ to be in $U_{i+1}$ all three of $xuv, xuw,$ and $xuz$ must be in $O_{i+1}$. In particular, in the first and third cases where $xvz \in \Gamma_{i+1}$, $z \in R_i(x,uv)$, and in the second and third cases where $uwz \in \Gamma_{i+1}$, $z \in R_i(u,xw)$. Further, in the first case where $uwz \in E_i$, $z \in S_i(u,xw)$ and similarly in the second case where $xvz \in E_i$ $z \in S_i(x,uv)$.

Then defining
\begin{align*}
    U_1 &:= \sum_{z \in R_i(x,uv) \cap S_i(u,xw)} \mathbbm{1}_{xvz \in \Gamma_{i+1}} \leq \sum_{z \in S_i(u,xw)} \mathbbm{1}_{xvz \in \Gamma_{i+1}} =: U_1^+ \\
    U_2 &:= \sum_{z \in S_i(x,uv) \cap R_i(u,xw)} \mathbbm{1}_{uwz \in \Gamma_{i+1}} \leq \sum_{z \in S_i(x,uv)} \mathbbm{1}_{uwz \in \Gamma_{i+1}} =: U_2^+ \\
    U_3 &:= \sum_{z \in R_i(x,uv) \cap R_i(u,xw)}\mathbbm{1}_{xvz, uwz \in \Gamma_{i+1}} \leq \sum_{z \in R_i(x,uv)} \mathbbm{1}_{xvz, uwz \in \Gamma_{i+1}} =: U_3^+
\end{align*}
we obtain
\begin{equation}
\notag
    |U_{i+1}(x,u,v,w) \setminus U_i(x,u,v,w)| \leq U_1 + U_2 + U_3 \leq U_1^+ + U_2^+ + U_3^+ =: X.
\end{equation}

Using  (\ref{eqn: event P_i^+}) yields
\begin{equation}
\label{eqn: U upper bound}
    |U_{i+1}(x,u,v,w)| \leq i(\log N)^9 + X.
\end{equation}

We first upper bound the expected value of $X$ using (\ref{eqn: event P_i}) and (\ref{eqn: U upper bound}) as
\begin{equation}
\label{eqn: U expected value}
    \mathbb{E}(X) \leq 4pq_i\pi_i \sqrt{N} + p^2q_i^2 N.
\end{equation}

Since $X$ is a sum of independent $(0,1)$ random variables, by Theorem \ref{standard chernoff} and (\ref{eqn: U upper bound}) and (\ref{eqn: U expected value})
\begin{align*}
    \mathbb{P}(|U_{i+1}(x,u,v,w)| \geq (i+1) \log^9 N) &\leq \mathbb{P}(X \geq \log^9 N) \\
    &\leq \exp \left(- \frac{(\log^9 N - 4pq_i\pi_i \sqrt{N} - p^2q_i^2 N)^2}{2(4pq_i \pi_i \sqrt{N} + p^2q_i^2 N)}\right) \\
    &\leq \exp \left(-\frac{(\log^9 N - 4\sigma q_i \pi_i - \sigma^2 q_i^2)^2}{8\sigma q_i \pi_i + \sigma^2 q_i^2} \right) \\
    &\leq \exp(-\log^{17} N) \\
    &\leq N^{-\omega(1)}.
\end{align*}

\subsection{Proof of $\mathcal{N}^+$}

The random variable $|N_{\Gamma_{i+1}}(vx) \cap A|$ is a sum of $|N_{O_i}(vx) \cap A|$ Bernoulli random variables each with probability $p$. Thus, 
\begin{equation}
\label{eqn: N^+ expected value}
\notag
\mathbb{E}(|N_{\Gamma_{i+1}}(vx) \cap A|) = p|N_{O_i}(vx) \cap A| \leq p|A|
\end{equation}
Applying Theorem \ref{standard chernoff} gives
\begin{align*}
    \mathbb{P}(|N_{\Gamma_{i+1}}(vx) \cap A| \geq p|A|(1 + N^{\frac{1}{4} + \beta})) &\leq \exp \left(-\frac{(p|A|N^{\frac{1}{4} + \beta})^2}{2p|A|}\right) \\
    &= \exp \left(-\frac{1}{2}p|A|N^{\frac{1}{2} + 2\beta}\right) \\
    &= \exp \left(-\frac{1}{2} \sigma |A| N^{2\beta}\right).
\end{align*}
By a union bound over all choices of $v, x$ and $A$, we obtain
\begin{align*}
    \mathbb{P}(\exists v, x, A: |N_{\Gamma_{i+1}(vx)} \cap A| \geq p|A|(1+N^{\frac{1}{4} + \beta})) &< N^{2+A} \exp \left(-\frac{1}{2} \sigma N^{2\beta} |A| \right) \\
    &=\exp \left((2+|A|)\log n -\frac{1}{2} \sigma N^{2\beta} |A| \right) \\
    &< N^{-\omega(1)}.
\end{align*}

\subsection{Proof of $\mathcal{Q}^+$}
Recalling that $s = n\sigma^4 q_I^2 = C\sqrt{N\log N}\sigma^4 q_I^2$, we begin with the case where $|A| = |B| \geq s$.
Since $O_{i+1} \subseteq O_i \backslash (C_{i+1}^1 \cup Y_{i+1})$

\begin{equation}
\label{eqn: Q^+ X def}
    |O_{i+1}(A,B,x)| \leq |O_i(A,B,x) \backslash (C_{i+1}^1 \cup Y_{i+1})| = \sum_{f \in O_i(A,B,x)} \mathbbm{1}_{f \not \in C_{i+1}^1 \cup Y_{i+1}} =: X
\end{equation}

We will first bound $\mathbb{E}(X)$ and then we will use Theorem \ref{chebyshev's inequality varraint} to prove concentration. Using $\mathcal{Q}_i^+$ and Lemma \ref{probability C^1 or Y}
\begin{align*}
    \mathbb{E}(X) &= |O_i(A,B,x)|\, \mathbb{P}(f \not\in C_{i+1}^1 \cup Y_{i+1}) \\
    &\leq q_i |A| |B| (\frac{q_{i+1}}{q_i} - 5\sigma^{3/2}q_i) \\
    &\leq q_{i+1}|A||B| - 5q_i^2 \sigma^{3/2}|A||B|.
\end{align*}

To bound the probability that $X \geq \mathbb{E}(X) + 5q_i^2\sigma^{3/2}|A||B|$, we will use Lemma \ref{chebyshev's inequality varraint}. For every $e \in O_i$, let $\Delta_e$ be the absolute change in $X$ when we change whether $e$ is in $\Gamma_{i+1}$. Similarly let $\Upsilon_e$ be the absolute change in $X$ when we change whether $e$ is in $Y_{i+1}$.

First, to bound $\Delta_e$, let $e = w_1w_2w_3$ and in this section assume $a \in A$ and $b \in B$. If $|\{w_1,w_2,w_3\} \cap (A \cup B \cup \{x\})| \leq 1$, then any $K_4^-$ which contains $e$ does not contain an edge of the form $xab$, so $\Delta_e = 0$.
Next, if $x \not\in \{w_1,w_2, w_3\}$ and $|\{w_1,w_2,w_3\} \cap (A \cup B)| \geq 2$, then if an edge of the form $xab$ is in a $K_4^-$ together with $e$, the vertices of this $K_4^-$ must be $\{w_1,w_2,w_3,x\}$ so $\Delta_e \leq 3$.

Finally, without loss of generality let $w_1 = x$ and assume $|\{w_2,w_3\} \cap (A \cup B)| \geq 1$. If $xab$ is an edge which is in $C_{i+1}^1$ only when $e \in \Gamma_{i+1}$, and without loss of generality $w_2 = b$ then either the edge $xw_3a$ or $w_2w_3a$ must be in $E_i$. The number of these edges can be bounded by $\sum_{j=1}^i (|N_{\Gamma_j}(xw_3) \cap A| + |N_{\Gamma_j}(w_2w_3) \cap A|)$. Because $\mathcal{N}_i^+$ holds, this is no more than $20p |A| N^{\frac{1}{4} + \beta} I$. Thus 

\begin{equation}
\label{eqn: Q^+ event Delta e bound}
\Delta_e \leq 40p|A|N^{\frac{1}{4} + \beta}I
\end{equation}

where the extra factor of $2$ comes from the possibility that both $w_2$ and $w_3$ are in $A \cup B$. Further, since changing whether $e \in \Gamma_{i+1}$ only changes if $xab$ is open if $xab \in \hat{S}_i(e)$, $\Delta_e \leq |\hat{S}_i(e) \cap O_i(A,B,x)|$. Then using (\ref{eqn: Q^+ event Delta e bound}) together with Lemma \ref{S_e intersection J}, we obtain

 \begin{equation}
 \notag
     p \sum_{e \in O_i} c_e^2 \leq 40p^2 |A|N^{\frac{1}{4} + \beta} I (6 q_i \pi_i \sqrt{N} (q_i|A|^2)) \leq 240 \sigma^2 N^{2\beta - \frac{1}{4}} q_i^2 \pi_i |A|^3.
 \end{equation}

 Further $\Upsilon_e \leq \mathbbm{1}_{e \in O_i(A,B,x)}$ so $$q_i \sum_{e \in O_i} \hat{c}_e^2 \leq q_i^2 |A|^2 \leq q_i^2 |A|^3 \left(\frac{1}{C\sigma^4 q_I^2}N^{-\frac{1}{2}} \log^{-\frac{1}{2}} N  \right) \leq \sigma^4 N^{2\beta - \frac{1}{4}} q_i^2 \pi_i |A|^3$$ where the second inequality comes from $|A| \geq s$. The last inequality holds, with room to spare, due to the power of $N$ which is $-1/2+2\beta+o(1)$ on the left and $-1/4+2\beta$ on the right. Hence

 \begin{equation}
     \label{eqn: Q^+ lambda}
     \lambda \leq p\sum_{e \in O_i} c_e^2 + q_i\sum_{e \in O_i} \hat{c}_e^2 \leq 250 \sigma^2 N^{2\beta - \frac{1}{4}}q_i^2 \pi_i |A|^3
 \end{equation}
Applying Theorem \ref{chebyshev's inequality varraint} gives
\begin{align*}
     \mathbb{P}(|O_{i+1}(A,B,x)| \geq q_{i+1} |A|^2) &\leq \mathbb{P}(X \geq \mathbb{E}(X) + q_i^2 \sigma^{3/2}|A|^2) \\
     &\leq \exp \left(-\frac{q_{i}^4 \sigma^3 |A|^4}{500\sigma^2 N^{2\beta-\frac{1}{4}}q_i^2\pi_i |A|^3}\right) \\
     &\leq \exp \left(-\frac{1}{500} q_i^3 \sigma |A| N^{\frac{1}{4} - 2\beta}\right) \\
     &\leq \exp \left(-\frac{1}{500} |A|N^{\frac{1}{4} - 6\beta} \right)
 \end{align*}
 where the third inequality uses (\ref{q and pi bounds, q times pi less than 1}) and the fourth inequality uses (\ref{eqn:sigmaqN}).
By the union bound over all choices of $x, A$ we obtain
 \begin{align*}
     \mathbb{P}(\exists x, A, B, |A|=|B|: |O_i(A,B,x)| \geq q_{i+1}|A||B|) &\leq N^{2|A|+1}\exp \left(-\frac{1}{500} |A| N^{\frac{1}{4} - 6\beta} \right) \\
     &\leq \exp \left((2|A|+1)\log N-\frac{1}{500} |A|N^{\frac{1}{4} - 6\beta}\right) \\
     &\leq N^{-\omega(1)}
 \end{align*}
 which proves the result when $|A| = |B|$.

 Now assume $|A| > |B| \geq s$. Then
 \begin{equation}
 \notag
     |O_i(A,B,x)| = \frac{\sum_{A' \subset A, |A'| = |B|} |O_i(A',B,x)|}{\binom{|A| - 1}{|B| - 1}}.
 \end{equation}

 Now applying the fact that  $\mathcal{Q}_i^+$  holds for  $|A| = |B|$ with probability at least $1 - N^{-\omega(1)}$ we obtain
 \begin{equation}
 \notag
     |O_i(A,B,x)| \leq \frac{\binom{|A|}{|B|} q_i |B|^2}{\binom{|A|-1}{|B|-1}} = q_i |A| |B|.
 \end{equation}
 with probability at least $1 - N^{-\omega(1)}$.

\subsection{Proof of $\mathcal{Q}$}

Recall we are to show that if  $A, B \in \binom{[N]}{n}$ are disjoint and $x \not\in A \cup B$, then $\tau_{i+1}q_{i+1} n^2 \leq |O_i(A,B,x)| \leq q_{i+1} |A||B|$. We begin with the following Lemma.

\begin{lemma}
    \label{Q event lemma}
    Let $\mathcal{Q}_{A,B,x}$ be the event that the following  hold where $x \in [N]$ and $A, B \in \binom{[N]}{n}$:
\begin{align*} X_1 &:= |O_i(A,B,x) \backslash (C^1_{i+1} \cup Y_{i+1})| \in \left[|O_i(A,B,x)|\left(\frac{q_{i+1}}{q_i} - 8\sigma^{3/2}q_i\right), |O_i(A,B,x)|\frac{q_{i+1}}{q_i} \right] \\
 X_2 &:= |O_i(A,B,x) \cap C^2_{i+1}| \leq |O_i(A,B,x)|(6\sigma^2 q_i) + 100\sigma q_i^2 \sqrt{N}n \\
   X_3 &:= |O_i(A,B,x) \cap \Gamma_{i+1}| \leq |O_i(A,B,x)| (2\sigma^2 q_i).
\end{align*}
   Then $\mathbb{P}(\neg \mathcal{Q}_{A,B,x} \cap \mathcal{N}_{i+1} \cap \mathcal{P}_{i+1}) \leq N^{-\omega(
    n)}$ for all choices of $A, B, x$.
\end{lemma}

Before proving Lemma \ref{Q event lemma}, we will use Lemma \ref{Q event lemma} to prove $\mathbb{P}(\neg \mathcal{Q}_{i+1} \cap \mathcal{N}_{i+1} \cap \mathcal{P}_{i+1}) \leq N^{-\omega(1)}$.  Assume for all $x, A, B$, the event $\mathcal{Q}_{A,B,x}$ holds. Let us show  that $\mathcal{Q}_{i+1}$ holds.

Let $A, B \in \binom{[N]}{n}$ and $x \in [N] \setminus (A \cup B)$. Recall that an edge in $O_{i}(A,B,x)$ remains in $O_{i+1}(A,B,x)$ if it is not in $C^1_{i+1} \cup Y_{i+1} \cup C^2_{i+1} \cup \Gamma_{i+1}$. Therefore
\begin{equation}
\notag
    X_1 - X_2 -X_3 \leq |O_{i+1}(A,B,x)| \leq X_1.
\end{equation}
For the upper bound, notice that since $\mathcal{G}_{\leq i}$ holds, $|O_i(A,B,x)| \leq q_i|A||B|$ by (\ref{eqn: event Q_i def}), so
\begin{equation}
\notag
    |O_{i+1}(A,B,x)| \leq X_1 \leq 
  |O_i(A,B,x)|\frac{q_{i+1}}{q_i} \le q_{i+1}|A||B|.
\end{equation}

For the lower bound, recall that $|A|=|B|=n=C\sqrt{N \log N}$ and $\sigma=1/\log^2 N$ and $\tau_i \leq 1$. Then
\begin{align*}
    X_1 - X_2 - X_3 &\geq |O_i(A,B,x)|\left(\frac{q_{i+1}}{q_i} - 8\sigma^{3/2}q_i -8\sigma^2q_i\right) -100\sigma q_i^2 \sqrt{N}n \\
    &\geq (\tau_i q_i n^2) \left(\frac{q_{i+1}}{q_i} - 16\sigma^{3/2}q_i\right) -100\sigma q_i^2 \sqrt{N}n \\
    & \geq n^2\left(\tau_i q_{i+1} - 16\sigma^{3/2} q_i^2 - \frac{100\sigma q_i^2}{C\sqrt{\log N}}\right) \\
    &=n^2\left(\tau_i q_{i+1} - 16\sigma^{3/2} q_i^2 - \frac{100\sigma^{5/4} q_i^2}{C}\right) \\
    & \geq q_{i+1} n^2 \left( \tau_i - \left(\frac{200 \sigma^{5/4} q_i}{C}\right)\left(\frac{q_i}{q_{i+1}}\right)\right)
\end{align*}
where in the second inequality we used (\ref{eqn: event Q_i def}) and in the equality we used $\sigma^{1/4} = \frac{1}{\sqrt{\log N}}$. Now recalling that $\tau_i = 1 - \frac{\delta \pi_i}{2\pi_I}$ and therefore $\tau_i - \tau_{i+1} = \frac{\delta \sigma q_i}{2\pi_I}$ together with $\sqrt{\frac{\log(\sqrt{3}x)}{3}} - \frac{1}{\sqrt{3}} \leq \Psi(x) \leq \sqrt{\frac{\log(\sqrt{3}x)}{3}} + \frac{1}{\sqrt{3}}$, $|\pi_i - \Psi(i\sigma)| \leq 2\sigma$, and $I = \lceil N^{\beta}\rceil$ from Lemma \ref{q and pi bounds} (\ref{Psi bound}) and (\ref{q and pi bounds pi close to psi}) gives

\begin{align*}
    \tau_{i+1} &= \tau_i - \frac{\delta \sigma q_i}{2\pi_I} \\
    &\leq \tau_i - \frac{\delta \sigma q_i}{2\left(\sqrt{\frac{\log (\sqrt{3}(I\sigma))}{3}} +(\frac{1}{\sqrt{3}} + 2\sigma)\right)} \\
    &\leq \tau_i - \frac{\delta \sigma q_i}{2.1\sqrt{ \frac{0.5 \log 3 + \beta \log N + \log \sigma}{3}}} \\
    &\leq \tau_i - \frac{\sqrt{3} \delta \sigma q_i}{(2.2)\sqrt{\beta \log N}} \\
    &= \tau_i - \frac{\sqrt{3}\delta q_i \sigma^{5/4}}{(2.2)\sqrt \beta}.
\end{align*}

 Now using $q_i/(q_{i+1}) \leq 12\sigma + 1 \leq 2$ from Lemma \ref{q and pi bounds} (\ref{q_i - q_i+1 < sigma min...}) and choosing constant $C$ so that $400/C \leq \frac{\sqrt{3}\delta}{(2.2)\sqrt{\beta}}$ gives $X_1 - X_2 - X_3 \geq q_{i+1}n^2 \tau_{i+1}$ as desired.
 
 We have shown that $\neg \mathcal{Q}_{i+1} \cap \mathcal{N}_{i+1} \cap \mathcal{P}_{i+1} \subset \cup_{x,A, B} \neg\mathcal{Q}_{A,B,x} \cap \mathcal{N}_{i+1} \cap \mathcal{P}_{i+1}$.
 Finally, applying Lemma~\ref{Q event lemma} and union bounding over all choices of $x, A, B$ gives 
 \begin{align*}
 \mathbb{P}(\neg \mathcal{Q}_{i+1} \cap \mathcal{N}_{i+1} \cap \mathcal{P}_{i+1}) &\leq \mathbb{P}(\exists x, A, B: \neg \mathcal{Q}_{A,B,x} \cap \mathcal{N}_{i+1} \cap \mathcal{P}_{i+1})\\
 &\leq N^{1+2n-\omega(n)} \\
 &\leq N^{-\omega(n)}.
 \end{align*}
 Now all that is needed to prove $\mathcal{Q}_{i+1}$ holds with high probability is  to prove Lemma \ref{Q event lemma}.

\subsection{Proof of Lemma \ref{Q event lemma}} \label{ss}

\subsubsection{$X_1$}
First we prove the bound on $X_1$. The upper bound follows from the proof for the event $\mathcal{Q}^+_i$ since $X_1$ is exactly the $X$ from (\ref{eqn: Q^+ X def}) and $n > s$. For the lower bound, we apply Lemma \ref{probability C^1 or Y} to show

\begin{equation}
    \label{eqn: X_1 lower bound expected value}
    \notag
    \mathbb{E}(X_1) = \sum_{e \in O_i{A,B,x}} \mathbb{P}(e \not \in C_{i+1,A,B,x}^1 \cup Y_{i+1}) \geq |O_{i}(A,B,x)|(\frac{q_{i+1}}{q_i} - 7\sigma^{3/2}q_i)
\end{equation}

Now using the same value for $\lambda$ from (\ref{eqn: Q^+ lambda}) and $C = 40p|A|N^{\frac{1}{4} + \beta} I$ which follows from (\ref{eqn: Q^+ event Delta e bound}) in the proof of $\mathcal{Q}^+$, we can apply the second half of Theorem \ref{chebyshev's inequality varraint} to show

$$  \mathbb{P}\left(X_1 \leq |O_i(A,B,x)|\left(\frac{q_{i+1}}{q_i} - 8\sigma^{3/2}q_i\right)\right) \leq \mathbb{P}(X_1 \leq \mathbb{E}(X_1) - \sigma^{3/2}q_i |O_i(A,B,x)|) \leq \exp \left(-D\right) 
$$
where 
$$D=\frac{(\sigma^{3/2}q_i |O_i(A,B,x)|)^2}{2((250 \sigma^2 N^{2\beta - \frac{1}{4}}q_i^2 \pi_i |A|^3) + (40p|A|N^{\frac{1}{4} + \beta} I)(\sigma^{3/2}q_i |O_i(A,B,x)|))}.$$
Note that 
\begin{align*}
D &\geq \frac{\sigma^3 \tau_i^2 q_i^4 n^4}{500 \sigma^2 N^{2\beta - \frac{1}{4}} q_i^2 \pi_i n^3 + 40\sigma^{5/2}N^{-\frac{1}{4} + 2\beta} n^3 q_i^2}\\
&\geq nN^{\frac{1}{4} - 10 \beta}
\end{align*}
where the last inequality uses (\ref{q and pi bounds, q times pi less than 1}) and (\ref{eqn:sigmaqN}) from Lemma \ref{q and pi bounds}. Since $e^{-D} = N^{-\omega(n)}$ we are done.

\subsubsection{$X_2$}
We now prove the bound on $X_2$. While one might expect that this would be relatively easy to show as it seems unlikely that two edges are picked in step $i$ which then form a $K_4^{-}$ with an edge of $O_{i}(A, b, x)$, this is the most difficult part of the proof. 

We will break up $X_2$ into  four components and handle each one separately. First we introduce some definitions and notation.
The following auxiliary variables  depend on the step of the process $i$, a distinguished vertex $x$, and two vertex subsets $A, B \subseteq [N] \setminus \{x\}$ with $|A|=|B|$.
\begin{equation}
    \label{eqn: z_A,B def}
    \notag
    z_{A,B}(i) := \sigma^4 q_i^2 |A|
\end{equation}
\begin{equation}
    \label{eqn: W_2 def}
    \notag
    W(A,B,x,i) := \{w \in [N] : |N_{\Gamma_{i+1}}(xw) \cap (A \cup B)| \geq z_{A,B}(i)\}
\end{equation}
\begin{equation}
\notag
    \label{eqn: C^2 hat def}
    \hat{C}_{i+1, A, B, x}^2 := \{xuv : \exists w \not \in W(A,B,x,i):  xuw, xvw \in \Gamma_{i+1}\}
\end{equation}
Whenever the context for $A$, $B$, $x$, and $i$ are clear, we will write $z$ for $z_{A,B}(i)$ and  $W$ for $W(A,B,x,i)$. We write $a$ (resp. $b$) for a generic vertex in $A$ (resp. $B$). 
\begin{align*}
    \hat{X_2} &:= |\{xab \in O_i(A,B,x) \cap C_{i+1}^2 : \exists w \not\in W(A,B,x) \text{ and } |\{xwa,xwb,wab\} \cap \Gamma_{i+1}| \geq 2\}|\\
    X_2^{AB} &:= |\{xab \in O_i(A,B,x) \cap C_{i+1}^2 : \exists w \in W(A,B,x) \text{ and } xwa,xwb \in \Gamma_{i+1}\}| \\
    X_2^{A} &:= |\{xab \in O_i(A,B,x) \cap C_{i+1}^2 : \exists w \in W(A,B,x) \text{ and } xwa,wab \in \Gamma_{i+1}\}| \\
    X_2^{B} &:= |\{xab \in O_i(A,B,x) \cap C_{i+1}^2 : \exists w \in W(A,B,x) \text{ and } xwb,wab \in \Gamma_{i+1}\}|.
\end{align*}

Since for $abx \in O_i(A,B,x) \cap C_{i+1}^2$ there must be some $w$ with $|\{xwa,xwb,wab\} \cap \Gamma_{i+1}| \geq 2$, we deduce that $X_2 \leq \hat{X_2} + X_2^{AB} + X_2^A + X_2^B$.

With the aim of applying Theorem \ref{Chernoff variant limited overlap}, we can upper bound $\hat{X_2}$ with 
\begin{equation}
    \label{eqn: hat X_2 upper bound}
    \notag
    \hat{X_2} = \sum_{e \in O_{i}(A,B,x)} \mathbbm{1}_{e \in \hat{C}^2_{i+1,A,B,x}} \leq \sum_{xab \in O_i(A,B,x)} \sum_{w \in V \setminus W(A,B,x,i)} \mathbbm{1}_{|\{axw, bxw, wab\} \cap \Gamma_{i+1}| \geq 2} =: \hat{X_2}^+.
\end{equation}

Now let
\begin{align*}
    L_{AB} &:= \{\{xwa,xwb\} \subseteq O_i: xab \in O_i, w \not \in \{x,a,b\}, a \in A, b \in B\} \\
    L_A &:= \{\{xwa,wab\} \subseteq O_i: xab \in O_i, w \not \in \{x,a,b\}, a \in A, b \in B\} \\   
    L_B &:= \{\{xwb,wab\} \subseteq O_i: xab \in O_i,  w \not \in \{x,a,b\}, a \in A, b \in B\}\\
    L &:= L_{AB} \cup L_A \cup L_B.
\end{align*}

Now with Theorem \ref{Chernoff variant limited overlap} in mind, for each $e \in O_i:=O$, let $\xi_{e}$ be the indicator random variable for $e \in \Gamma_{i+1}$. 
For $\alpha \in L$, let $Y_{\alpha} = \mathbbm{1}_{\xi_{e}=1 \forall e \in \alpha} = \mathbbm{1}_{\alpha \subseteq \Gamma_{i+1}}$. In order to bound $\sum_{\alpha \in L} \mathbb{E}(Y_{\alpha})$, note that if $\{xwa, xwb\} \in L_{AB}$ then $w \in R_{i}(x,ab)$. This shows that 
$$|L_{AB}| \le |O_i(A,B,x)||R_{i}(x,ab)|,$$
since $xab \in O_i(A, B, x)$ and $w \in R_i(x, ab)$ means that $xwa$ and $xwb$ are open.

Similarly, if $\{xwa, wab\} \in L_A$, then $w \in R_{i}(a,xb)$ and if $\{xwb, wab\} \in L_B$ then $w \in R_i(b,xa)$. Thus $$|L| \leq |O_i(A,B,x)|(|R_{i}(x,ab)| + |R_{i}(a,xb)| + |R_{i}(b,xa)|).$$ Since $|R_{i}(u_1,u_2u_3)| \leq q_i^2 N$ 
by (\ref{eqn: event P_i}) due to $\mathcal{P}_i$, and $q_i \leq 1$, we can give the upper bound
\begin{equation}
    \label{eqn: hat X_2 sum Y alpha bound}
    \notag
    \sum_{\alpha \in L}\mathbb{E}(Y_{\alpha}) \leq p^2|O_{i}(A,B,x)| \left(3\max \{|R_{i}(x,ab)|\}\right) \leq 3\sigma^2 q_i |O_i(A,B,x)| =: \mu.
\end{equation}
Now we define $K \subseteq L$ as follows:
\begin{align*}
    K_{AB} &:= \{\{xwa,xwb\} \in L_{AB} : w \not\in W(A,B,x,i) \text{ and } xwa, xwb \in \Gamma_{i+1}\} \\
    K_{A} &:= \{\{xwa,wab\} \in L_A : w \not\in W(A,B,x,i) \text{ and } xwa,wab \in \Gamma_{i+1}\} \\
    K_B &:= \{\{xwb,wab\} \in L_B  : w \not\in W(A,B,x,i) \text{ and } xwb, wab \in \Gamma_{i+1}\} \\
    K &:= K_{AB} \cup K_{A} \cup K_{B}.
\end{align*}

Note that $\hat{X}_2^+  = \sum_{\alpha \in K} Y_{\alpha}$. Next, recalling from Theorem \ref{Chernoff variant limited overlap} that $f(J) := \max_{\beta \in J} |\alpha \in J : \alpha \cap \beta \neq \emptyset|$, we claim that $f(K) \le 8z$. Then since in Theorem \ref{Chernoff variant limited overlap}, $Z_{8z}$ is defined as the maximum over $J \subseteq L$ with $f(J) \leq 8z$ of $\sum_{\alpha \in J} Y_{\alpha}$, if $f(K) \leq 8z$ then $\hat{X}_2^+ \leq Z_{8z}$. Finally, we apply Theorem~\ref{Chernoff variant limited overlap} to conclude using $t=\mu$ that 
$$ \mathbb{P}(\hat{X_2} \geq 6\sigma^2 q_i |O_i(A,B,x)|) \le \mathbb{P}(Z_{8z}\ge 2\mu)\le \exp(-\mu^2/(32z\mu))= \exp(-\mu/(32z)).$$
We now provide the details of these assertions.

We now show that $f(K) \leq 8z$. Let $\beta \in K$ and let $e \in \beta$.  First, if $x\not\in e$ notice that there are at most $3$ elements $\alpha \in K$ such that $e \in \alpha$. Further, if $x \in e$, then the number of elements $\alpha \in K$ such that $e \in \alpha$ is bounded by $|N_{\Gamma_{i+1}}(xw) \cap (A \cup B)|$, and since for all $\beta \in K$, $w \not\in W(A,B,x,i)$, we obtain $|N_{\Gamma_{i+1}(xw)} \cap (A \cup B)| < z$. Thus for all $\beta \in K$
 \begin{equation}
 \notag
  |\alpha \in K : \alpha \cap \beta \neq \emptyset| \leq \sum_{e \in \beta} |\alpha \in K : e \in \alpha| \leq \sum_{e \in \beta} 3 + \sum_{v \in e \setminus W(A,B,x,i)} |N_{\Gamma_{i+1}}(xv) \cap (A \cup B)| \leq 8z.
\end{equation}

Now using Theorem \ref{Chernoff variant limited overlap} since $\hat{X_2}^+ \leq \sum_{\alpha \in K} Y_{\alpha} \leq Z_{8z}$, set $\mu = 3\sigma^2 q_i |O_i(A,B,x)|$, and $C = 8z$. Using the fact that $|O_i(A,B,x)| \geq \tau_iq_in^2$ by (\ref{eqn: event Q_i def}) since we are in $\mathcal{G}_{\leq i}$, we obtain 
\begin{align*}
    \label{eqn: hat X_2 probability concentration}
    \mathbb{P}(\hat{X_2} \geq 6\sigma^2 q_i |O_i(A,B,x)|) &\leq \mathbb{P}(Z_{8z} \geq 2\mu) \\
    &\leq \exp \left(-\frac{\mu^2}{32\mu z}\right) \\
    & \leq \exp \left(-\frac{3\tau_i q_i^2 \sigma^2 n^2}{32\sigma^4 q_i^2 n} \right) \leq N^{-\omega(n)}.
\end{align*}

Next we prove the bound for $X_2^{AB}$. By definition
\begin{equation}
\notag
    X_2^{AB} \leq \sum_{w \in W(A,B,x,i)} |O_i(N_{\Gamma_{i+1}}(wx) \cap A, N_{\Gamma_{i+1}}(wx) \cap B,x)|.
\end{equation}

If $$\min\{|N_{\Gamma_{i+1}}(wx) \cap A|, |N_{\Gamma_{i+1}(wx)} \cap B|\} \geq z=\sigma^4q_i^2|A| \ge n\sigma^2 q_I^2= s,$$
then we use the fact that $\mathcal{Q}^+_{i}$ holds and obtain
 $$O_i(N_{\Gamma_{i+1}}(wx) \cap A, N_{\Gamma_{i+1}(wx)} \cap B, x) \leq q_i|N_{\Gamma_{i+1}}(wx) \cap A||N_{\Gamma_{i+1}(wx)} \cap B|.$$
 Otherwise  one of these terms is less than $z$. Writing
$$M:=\max\{|N_{\Gamma_{i+1}}(wx) \cap A|, |N_{\Gamma_{i+1}}(wx) \cap B|\}, $$
this gives
$$  |O_i(N_{\Gamma_{i+1}}(wx \cap A), N_{\Gamma_{i+1}}(wx) \cap B,x)|   \leq  z M.$$
Putting both bounds together we obtain
\begin{equation}
\label{eqn: Q event X_AB auxillary}
\begin{split}
   |O_i(N_{\Gamma_{i+1}}(wx \cap A), N_{\Gamma_{i+1}}(wx) \cap B,x)| 
  &\leq q_i|N_{\Gamma_{i+1}}(wx) \cap A| |N_{\Gamma_{i+1}}(wx) \cap B| + z M\\
    &\leq \left(q_i|N_{\Gamma_{i+1}}(wx)| + z \right) |N_{E_i \cup \Gamma_{i+1}}(wx) \cap (A \cup B)|.
\end{split}
\end{equation}
We will now use the following Lemma from~\cite{GW} which bounds the sum of the sizes of a collection of sets whose intersection is bounded. We will provide a proof as well for completeness.

\begin{lemma}
    \label{Intersecting Sets Lemma}
    Let $(U_i)_{i \in I}$ be a family of subsets $U_i \subseteq U$, and suppose $|U_i| \geq z \ge  \sqrt{4|U|y}$ and $|U_i \cap U_j| \leq y$ for $i \neq j$. Then $|I| \leq \frac{2|U|}{z}$ and $\sum_{i \in I} |U_i| \leq 2|U|$.
\end{lemma}

\begin{proof}
    Suppose for contradiction $|I| > \frac{2|U|}{z}$ and let $J \subseteq I$ with $|J| = \lfloor \frac{2|U|}{z}\rfloor + 1$. Note that $y \leq \frac{z^2}{4|U|}$ and $\frac{z}{2} \leq \frac{|U_i|}{2}$ for all $i \in I$. Then for any $i \in J$
    \begin{equation}
        \notag
        \sum_{j \in J, j\neq i} |U_i \cap U_j| \leq (|J| - 1)y \leq \frac{2|U|y}{z} \leq \frac{z}{2} \leq \frac{|U_i|}{2}.
    \end{equation}

Then
\begin{equation}
    \notag
    |U| \geq |\cup_{i \in J} U_i|  \geq \sum_{i \in J} \left(|U_i| - \sum_{j \in J: j\neq i} |U_i \cap U_j|\right) \geq \sum_{i \in J} \frac{|U_i|}{2} \geq |J|\frac{z}{2} = \left(\lfloor\frac{2|U|}{z}\rfloor + 1\right)\frac{z}{2} > |U|
\end{equation}
which is a contradiction. Hence $|I| \leq \frac{2|U|}{z}$. Now for any $i \in I$
\begin{equation}
    \notag
    \sum_{j \in I, j \neq i} |U_i \cap U_j| \leq (|I| - 1)y \leq \frac{2|U|y}{z} \leq \frac{z}{2} \leq \frac{|U_i|}{2}
\end{equation}

as before and 
\begin{equation}
    |U| \geq |\cup_{i \in I} U_i| \geq \sum_{i \in I} \left(|U_i| - \sum_{j \in I, j \neq i} |U_i \cap U_j|\right) \geq \sum_{i \in I} \frac{|U_i|}{2}.
\end{equation}
\end{proof}

We will let $W(A,B,x,i) = I$, $A \cup B = U$, $U_w = N_{E_i \cup \Gamma_{i+1}}(wx) \cap (A \cup B)$ in Lemma \ref{Intersecting Sets Lemma}. As $\mathcal{P}_{i+1}$ holds, for all $w_1, w_2 \in W$,
\begin{align*}
    |(N_{E_i \cup \Gamma_{i+1}}(w_{1}x) \cap (A \cup B)) \cap (N_{E_i \cup \Gamma_{i+1}}(w_2x) \cap (A \cup B))| &\leq |N_{E_{i+1}}(w_1x) \cap N_{E_{i+1}}(w_2x)| \\&\leq T_{i+1}(x,w_1w_2) \\
    &\leq I(\log N)^9.
\end{align*}

Let $y := I (\log N)^9$ and note that $z = \sigma^4 q_i^2 n \geq \sqrt{4 (2n)(I (\log N)^9)}$. Also since for all $w \in W(A,B,x,i)$ we have  $|N_{\Gamma_{i+1}}(xw) \cap (A \cup B)| \geq z$, then $|U_w| \geq z$. Thus Lemma \ref{Intersecting Sets Lemma} gives

\begin{equation}
\notag
    \sum_{w \in W(A,B,x,i)} |N_{E_i \cup\Gamma_{i+1}}(wx) \cap (A \cup B)| \leq 4n.
\end{equation}

Since $\mathcal{N}_{i+1}$ holds,  for all $w \in W(A,B,x,i)$, we have  $|N_{\Gamma_{i+1}}(wx)| \leq 2q_i\sigma \sqrt{N}$. Also,  $z = \sigma^4 q_i^2 n \ll 2q_i^2\sigma \sqrt{N}$, so 
$$q_i|N_{\Gamma_{i+1}}(wx)| + z \le 3q_i^2\sigma \sqrt{N}.$$
Putting (\ref{eqn: Q event X_AB auxillary}) and Lemma \ref{Intersecting Sets Lemma} together, we obtain
\begin{equation}
\notag
    X^{AB} \leq  \sum_{w \in W(A, B, x, i)} \left(q_i|N_{\Gamma_{i+1}}(wx)| + z \right) |N_{E_i \cup \Gamma_{i+1}}(wx) \cap (A \cup B)| \le (3q_i^2\sigma \sqrt{N})(4n).
\end{equation}

The last bound we need to prove to bound $X_2$ is a bound on $X^A$ and $X^B$. We will show the bound on $X^A$ and the argument for $X^B$ is symmetric. Since $\mathcal{N}^+$ holds, $|N_{\Gamma_{i+1}}(aw) \cap B| \leq 10p|B|(N^{\frac{1}{4} + \beta})$ for all $a$ and $w$. Thus we can bound $X_2^A$ by
\begin{align*}
    X_2^A &\leq \sum_{w \in W(A,B,x,i)} \sum_{a \in N_{\Gamma_{i+1}}(xw) \cap A} |N_{\Gamma_{i+1}}(aw) \cap B| \\
    &\leq \left(10pnN^{\frac{1}{4} + \beta} \right) \sum_{w \in W(A,B,x,i)} |N_{\Gamma_{i+1}}(wx) \cap A| \\
    &\leq 10C\sigma \sqrt{\log N}N^{\frac{1}{4} + \beta} \sum_{w \in W(A,B,x,i)} |N_{E_{i+1}}(wx) \cap (A \cup B)|.
\end{align*}

Now we apply Lemma \ref{Intersecting Sets Lemma} with $I = W(A,B,x,i)$, $U = A \cup B$, and $U_w = N_{E_{i+1}}(wx) \cap (A \cup B)$. Again, $|U_w| \geq z$ since $w \in W(A,B,x,i)$ and $\Gamma_{i+1} \subseteq E_{i+1}$, and $|U_{w_1} \cap U_{w_2}| \leq |T_{i+1}(x,w_1w_2)| \leq I(\log N)^9 =: y$. Also, $z \geq \sqrt{4 n y}$, so Lemma~\ref{Intersecting Sets Lemma} yields
$$\sum_{w \in W(A,B,x,i)} |N_{E_{i+1}}(wx) \cap (A \cup B)| \le 4n.$$
Altogether,
\begin{equation}
\notag
    X_2^A \leq 10C \sigma \sqrt{\log N} N^{\frac{1}{4} + \beta} (4n).
\end{equation}

Therefore, whp, 
\begin{align*}
X_2 &\leq \hat{X_2} + X_2^{AB} + X_2^A + X_2^B \\
&\leq 6\sigma^2q_i|O_i(A,B,x)| + 3q_i^2\sigma\sqrt{N}(4n) + 2(10C\sigma \sqrt{\log N})N^{\frac{1}{4} + \beta}(4n) \\
&\le 6\sigma^2q_i|O_i(A,B,x)| + 100\sigma q_i^2\sqrt N n.
\end{align*}

\subsubsection{$X_3$}
Finally, we prove the bound on $X_3$. Since $X_3$ is a sum of iid Bernoulli random variables each with probability $p$, 
\begin{equation}
\notag
    \mathbb{E}(X_3) = p|O_i(A,B,x)| = \frac{\sigma}{\sqrt{N}} |O_i(A,B,x)| \ll \sigma^2 q_i |O_i(A,B,x)| =: t,
\end{equation}
where the last inequality comes from Lemma \ref{q and pi bounds} (\ref{eqn:sigmaqN}). Applying Theorem \ref{standard chernoff} together with the bound $|O_i(A,B,x)| \geq \tau_i q_i |A| |B|$ since $\mathcal{G}_{\leq i}$ holds gives
\begin{equation}
\notag
    \mathbb{P}(X_3 > |O_i(A,B,x)|(2\sigma^2 q_i)) \leq \mathbb{P}(X_3 \geq \mathbb{E}(X_3) + t) \leq \exp\left(-\frac{t^2}{4t}\right) \le  \exp\left(-\frac{\sigma^2 q_i^2 \tau_i |A| |B|}{4}\right).
\end{equation}
By   Lemma \ref{q and pi bounds} (\ref{eqn:sigmaqN}), $|A| = |B| = C \sqrt{N \log N}$ and $\tau_i \geq 1 - \frac{\delta}{2}$, we have $\mathbb{P}(X_3 > |O_i(A,B,x)|2\sigma^2 q_i) \leq N^{-\omega(n)}$.

\section{Bound on $\neg \mathcal{T}$}

To complete the proof of Lemma \ref{secondary_lemma}, recall that $B_{i+1}$ is a collection of bad subsets of $\Gamma_{i+1}$ and $H_{i+1}$ is comprised of the edges selected in $\Gamma_{i+1}$ which remain after removing the edges  from $D_{i+1}$, which is a maximal subset of $B_{i+1}$. Rather than conditioning on $\mathcal{G}_{\leq i}$ as we did in previous sections, here we fix some $x \in [N]$ and disjoint sets $A$ and $B$ with $|A| = |B| = n$ and we consider the number of edges of the form $xab$ added over the entire process where $a \in A$ and $b \in B$. We break this quantity up into those added into $E_i$ and those which remain in $H_i$ by defining
\begin{equation}
\notag
    X(A,B,x) := \sum_{i=0}^{I-1} |O_i(A,B,x) \cap \Gamma_{i+1}|
\end{equation}

\begin{equation}
\notag
    Y(A,B,x) := \sum_{i=0}^{I-1} |O_i(A,B,x) \cap \left( \cup_{S \in D_{i+1}} \cup_{e \in S} e \right)|.
\end{equation}

We now define $\mu^+$ and $\mu^-$ in order to bound $X(A,B,x)$ and $Y(A,B,x)$ as

\begin{equation}
    \notag
    \mu^+ := p\sum_{i=0}^{I-1} q_i n^2
\end{equation}

\begin{equation}
    \notag
    \mu^- := p\sum_{i=0}^{I-1} \tau_i q_i n^2.
\end{equation}

Since $H$ only includes edges picked in $E_i$ which were not in $D_i$ for some $i$
\begin{equation}
\label{T event H in [X - Y, X]}
    X(A,B,x) - Y(A,B,x) \leq e_{H,x}(A,B) \leq X(A,B,x).
\end{equation}

Thus if we prove

\begin{equation}
\label{T event X bound goal}
    X(A,B,x) \in \left[\left(1 - \frac{\delta}{2}\right)\mu^-, \left(1 + \frac{\delta}{2} \right)\mu^+\right]
\end{equation}

\begin{equation}
\label{T event Y bound goal}
    Y(A,B,x) \leq \frac{\delta^2 \mu^-}{9}
\end{equation}

then (\ref{T event H in [X - Y, X]}), (\ref{T event X bound goal}), and (\ref{T event Y bound goal}) and $0 < \delta < 1$ will show

\begin{equation}
    \label{e_H mu^- and mu^+ bound}
    \left(1-\frac{\delta}{2} - \frac{\delta^2}{9}\right) p\sum_{i=0}^{I-1} \tau_iq_in^2 \leq e_{H,x}(A,B) \leq \left(1 + \frac{\delta}{2} \right)p\sum_{i=0}^{I-1}q_in^2.
\end{equation}

Observe that  (\ref{eqn: pi def}) implies $$p\sum_{i=0}^{I-1} q_i n^2 =
\frac{n^2}{\sqrt{N}}\sum_{i=0}^{|I|-1}\sigma q_i
=\frac{n^2}{\sqrt{N}}(\pi_I - \sigma).$$

By (\ref{Psi bound}) and (\ref{q and pi bounds pi close to psi}) 
\begin{equation}
\label{Pi combined bound}
\pi_I \in \left[\sqrt{\frac{\log (\sqrt{3}) + \beta \log N}{3}} - \frac{1}{\sqrt{3}} + \sigma, \sqrt{\frac{\log (\sqrt{3}) + \beta \log N}{3}} +\frac{1}{\sqrt{3}} + 2\sigma\right].
\end{equation}

 Since $\rho = \sqrt{\frac{\beta \log N}{3N}}$ we obtain 
 $$e_{H,x}(A,B) \leq \left(1 + \frac{\delta}{2} \right)p\sum_{i=0}^{I-1}q_in^2 \le 
 \left(1 + \frac{\delta}{2} \right)
 \frac{n^2}{\sqrt{N}}(\pi_I - \sigma)
 < (1+\delta) \rho n^2.$$ Similarly since $\tau_I = 1- \frac{\delta}{2}$ and $\mu^- \geq \tau_I\mu^+$ then (\ref{e_H mu^- and mu^+ bound}) gives $e_{H,x}(A,B) \geq (1 - \delta + \frac{5\delta}{36})\mu^+$. Then by a similar argument to the upper bound $e_{H,x}(A,B) \geq (1 - \delta) \rho n^2$.
Hence, showing (\ref{T event X bound goal}) and (\ref{T event Y bound goal}) hold for all $x, A, B$ with probability $1 - N^{-\omega(1)}$ will complete the proof of Lemma \ref{secondary_lemma}.

To bound $X(A,B,x)$, we define
\begin{equation}
\notag
    X^+_{i+1} := \mathbbm{1}_{\mathcal{G}_i} \sum_{e \in O_i(A,B,x)}\mathbbm{1}_{e \in \Gamma_{i+1}}
\end{equation}
and
\begin{equation}
\notag
    X^+ := \sum_{i=0}^{I-1} X^+_{i+1}.
\end{equation}

Next define $Z^+_{i+1} = \text{Bin}\left(q_i n^2, p \right)$ to be independent random variables. As $\mathbb{P}(\neg \mathcal{G}_{i}) = N^{-\omega(1)}$ and we are in the event $\mathcal{G}_{i}$ we have $|O_i(A,B,x)| \leq q_i n^2$. Consequently, $\mathbb{P}(X^+_{i+1} \geq t) \leq \mathbb{P}(Z^+_{i+1} \geq t)$ for all $t$ and for all $i$. Then let $Z^+ = \sum_{i=0}^{I-1} Z^+_{i+1}$ and notice that this is equal in distribution to $\text{Bin}\left(\sum_{i=0}^{I-1} q_i n^2, p \right)$. Also, for all $t$
\begin{equation}
\notag
    \mathbb{P}(X(A,B,x) \geq t \text{ and } \mathcal{G}_{\leq I}) \leq \mathbb{P}(X^+ \geq t) \leq \mathbb{P}(Z^+ \geq t).
\end{equation}
Similarly, defining $Z^-_{i+1} := \text{Bin}\left( \tau_i q_i n^2, p \right)$ and $Z^- := \text{Bin}\left(\sum_{i=0}^{I - 1} \tau_i q_i n^2, p \right)$ gives 
\begin{equation}
\notag
    \mathbb{P}(X(A,B,x) \leq t \text{ and } \mathcal{G}_{\leq I}) \leq \mathbb{P}(Z^- \leq t).
\end{equation}

Let $\mu^+ = \mathbb{E}(Z^+) = p\sum_{i=0}^{I-1}q_i n^2$ and $\mu^- = \mathbb{E}(Z^-) = p\sum_{i=0}^{I-1} \tau_i q_i n^2$. Now we can apply Theorem \ref{standard chernoff} to $Z^+$ and $Z^-$ to get

\begin{equation}
\notag
    \mathbb{P}\left(X(A,B,x) \geq (1 + \frac{\delta}{2})\mu^+ \text{ and } \mathcal{G}_{\leq I}\right) \leq \mathbb{P}\left(Z^+ \geq (1+\frac{\delta}{2})\mu^+\right) \leq \exp \left( -\frac{\delta^2 \mu^+}{8}\right)
\end{equation}

\begin{equation}
\notag
    \mathbb{P}\left(X(A,B,x) \leq (1 - \frac{\delta}{2})\mu^- \text{ and } \mathcal{G}_{\leq I}\right) \leq \mathbb{P}\left(Z^- \leq (1 - \frac{\delta}{2})\mu^- \right) \leq \exp \left(- \frac{\delta^2\mu^-}{8(1 + \frac{\delta}{2})} \right).
\end{equation}

By (\ref{eqn: pi def}) $$\mu^+ = \frac{n^2}{\sqrt{N}} \sum_{i=0}^{I-1} \sigma q_i = \frac{n^2}{\sqrt{N}}(\pi_{I} - \sigma).$$ Using the bound on $\pi_I$ from (\ref{Pi combined bound}) gives $\mu^+ \geq \frac{C}{\sqrt{3}}\sqrt{\log N} n (\sqrt{\beta\log N} -1)$, and similarly since $\tau_i = 1 - \frac{\delta \pi_i}{2\pi_I} \geq 1 - \frac{\delta}{2}$ and $0 \leq \delta \leq 1$, then $\mu^- \geq \frac{C}{2\sqrt{3}}\sqrt{\log N} n (\sqrt{\beta \log N} - 1)$. Thus 

\begin{equation}
\label{T event X(A,B,x) bound}
\notag
\mathbb{P}\left(X(A,B,x) \not \in [(1-\frac{\delta}{2})\mu^-, (1+\frac{\delta}{2})\mu^+] \text{ and } \mathcal{G}_{\leq I} \right) \leq 2\exp \left(-\frac{C \sqrt{B} \delta^2}{64\sqrt{3}} n\log N  \right).
\end{equation}
By the union bound
\begin{align*}
    &\mathbb{P}\left(\exists x, A, B: X(A,B,x) \not \in \left[\left( 1 - \frac{\delta}{2}\right)\mu^-, \left(1 + \frac{\delta}{2}\right)\mu^+\right] \text{ and } \mathcal{G}_{\leq I}\right) \\
    &\leq 2N^{2n+1}\exp \left(-\frac{C \sqrt{B} \delta^2}{64\sqrt{3}} n\log N  \right) \\
    &\leq 2\exp \left((2n+1)\log N-\frac{C \sqrt{B} \delta^2}{64\sqrt{3}} n\log N  \right) \\
    &\leq N^{-\omega(1)}
\end{align*}

where the last inequality follows if we let $C > \frac{144 \sqrt{3}}{\delta^2 \sqrt{\beta}}$.

Now turning to $Y(A,B,x)$ and recalling that $B_{i+1}$ is the set of subsets of $\Gamma_{i+1}$ which create a $K_4^-$ together with $H_i$ and $D_{i+1}$ is a maximal subset of $B_{i+1}$, we introduce $$\hat{Y}(A,B,x) := \sum_{i=0}^{I-1} |O_i(A,B,x) \cap \left(\cup_{S \in B_{i+1}} \cup_{e \in S} e \right)|.$$

Let 
\begin{align*}
Y_{i+1}(A,B,x) &:= |O_i(A,B,x) \cap \left( \cup_{S \in D_{i+1}} \cup_{e \in S} e\right)|, \\ \hat{Y}_{i+1}(A,B,x) &:= |O_i(A,B,x) \cap \left( \cup_{S \in B_{i+1}} \cup_{e \in S} e\right)|. 
\end{align*}

Let $y := \frac{\delta^2 \mu^-}{9}$. Writing $$Y^I_y=\{{\bf y} =(y_1, \ldots, y_{I}) \in \mathbb N^I: \sum_{i=1}^I y_i=y\},$$  

\begin{equation}
\label{eqn: Y(A,B) setup}
\begin{split}\mathbb{P}(Y(A,B,x) \geq y \text{ and } \mathcal{G}_{\leq I}) &\leq \sum_{{\bf y} \in \mathbb{Y}^I_y} \mathbb{P}(\cap_{i=0}^{I-1}Y_{i+1}(A,B,x) \geq y_{i+1} \text{ and } \mathcal{G}_{\leq i+1})\\
&= \sum_{{\bf y} \in \mathbb{Y}^I_y} \prod_{i=0}^{I-1} \mathbb{P}(Y_{i+1} \geq y_{i+1} | \cap_{j=0}^{i-1} Y_{j+1}(A,B,x) \geq y_{j+1} \text{ and } \mathcal{G}_{\leq j+1}).
\end{split}
\end{equation}

Here we will use the concentration inequality Theorem \ref{maximal subset bound} to bound $Y_{i+1}(A,B,x)$. To this end we define
\begin{align*}
Y_{i+1}^+(A,B,x) &:= \sum_{e \in O_i(A,B,x)} \left(\sum_{\{e',e'' \in O_i : ee'e'' \text { is a } K_4^-\}} \mathbbm{1}_{ee'e'' \in D_{i+1}} + \sum_{\{e' \in O_i, e'' \in E_i : ee'e'' \text{ is a } K_4^- \}}  \mathbbm{1}_{ee' \in D_{i+1}}\right), \\
\hat{Y}_{i+1}^+(A,B,x) &:= \sum_{e \in O_i(A,B,x)} \left(\sum_{\{e',e'' \in O_i : ee'e'' \text { is a } K_4^-\}} \mathbbm{1}_{ee'e'' \in B_{i+1}} + \sum_{\{e' \in O_i, e'' \in E_i : ee'e'' \text{ is a } K_4^- \}}  \mathbbm{1}_{ee' \in B_{i+1}}\right),
\end{align*}
and note that $Y_{i+1}(A,B,x) \leq Y_{i+1}^+(A,B,x)$ and $\hat{Y}_{i+1}(A,B,x) \leq \hat{Y}_{i+1}^+(A,B,x)$.
    
Recall that $B_{i+1}^2$ is the set of pairs of edges $\{e,e'\}$ added to $\Gamma_{i+1}$ where there is a third edge $e'' \in H_i \subseteq E_i$ so $e,e',e''$ comprise a $K_4^-$, and $B^3_{i+1}$ is the set of triples of edges added to $\Gamma_{i+1}$ which make a $K_4^-$. Further, recalling that $B_{i+1} = B^2_{i+1} \cup B^3_{i+1}$, $\hat{Y}_{i+1}^+(A,B,x)$ includes edges from sets in $B_{i+1}^2$ and edges from sets $B_{i+1}^3$. If $xab$ is the edge in $\hat{Y}_{i+1}(A,B,x)$ then the number of $w \in [N]$ which can be the fourth vertex of the $K_4^-$ in the case where one edge from the $K_4^-$ is in $E_i$ is $|S_{i}(x,ab)| + |S_{i}(a,bx)| + |S_{i}(b,ax)| \leq 6q_i\pi_i \sqrt{N}$ assuming $\mathcal{G}_{\leq i}$. Similarly the number of $w$ that can be the fourth vertex from a $K_4^-$ where none of the edges were in $E_i$ is $|R_{i}(x,ab)| + |R_{i}(a,xb)| + |R_{i}(b,ax)| \leq 3q_i^2N$ assuming $\mathcal{G}_{\leq i}$. Hence assuming $\mathcal{G}_{\leq i}$
\begin{equation}
\notag
\begin{split}
    \mathbb{E}(\hat{Y}_{i+1}^+(A,B,x)) &\leq |O_i(A,B,x)|\left(\mathbb{P}\left(xab \in (\cup_{S \in B_{i+1}^2} \cup_{e \in S} e)\right) + \mathbb{P}\left(xab \in (\cup_{S \in B_{i+1}^3} \cup_{e \in S} e) \right) \right)\\
    &\leq q_i n^2 p^2 \left(6q_i\pi_i \sqrt{N}\right) + q_i n^2 p^3 \left(3q_i^2 N \right).
    \end{split}
\end{equation}
 By Lemma \ref{q and pi bounds} we have $\max \{q_i^2, q_i\pi_i\} \leq 1$. Recalling $p= \sigma/\sqrt N$ we obtain
$$\mathbb{E}(\hat{Y}_{i+1}^+(A,B,x)) \leq  10\sigma p q_i n^2 =:\mu_{i+1}.$$ 
Our plan now is to apply  Theorem \ref{maximal subset bound}.
To this end, let $Q= O_i$, $\xi_i = \mathbbm{1}_{e \in \Gamma_{i+1}}$,
\begin{align*}
    J^2 &:= \{(e,e'): e \in O_i(A,B,x), e' \in O_i, \exists e'' \in E_i \text{ s.t. } ee'e'' \text{ is a } K_4^-\} \\
    J^3 &:= \{(e,e',e''): e \in O_i(A,B,x), e' \in O_i, e'' \in O_i, ee'e'' \text{ is a } K_4^-\} \\
    J &:= J^2 \cup J^3.
\end{align*}
For all $\alpha \in J^2$ let $Q(\alpha) = \{e,e'\}$ and for $\alpha \in J^3$ let $Q(\alpha) = \{e,e',e''\}$. Now $X_0 = Y_{i+1}^+(A,B,x)$ and $X = \hat{Y}_{i+1}^+(A,B,x)$ so by Theorem \ref{maximal subset bound}
\begin{equation}
\notag
    \mathbb{P} \left(Y_{i+1}(A,B,x) \geq y_{i+1} \right) \leq \mathbb{P}(Y_{i+1}^+(A,B,x) \geq y_{i+1}) \leq  \frac{\left(\mathbb{E}(\hat{Y}_{i+1}^+(A,B,x))\right)^{y_{i+1}}}{y_{i+1}!} \leq \left(\frac{\mu_{i+1}e}{y_{i+1}} \right)^{y_{i+1}}
    \end{equation}
where the last inequality comes from Stirling's formula.

Then for all $y_{i+1}$ such that $y_{i+1} \geq \frac{10\mu_{i+1}}{\sqrt{\sigma}}$ and assuming $\mathcal{G}_{\leq I}$ we get 
\begin{equation}
\label{eqn: Y(A,B) large y}
\mathbb{P}(Y_{i+1}(A,B,x) \geq y_{i+1}) \leq \sigma^{\frac{y_{i+1}}{2}}.
\end{equation}

Turning to $y_{i+1} < \frac{10\mu_{i+1}}{\sqrt{\sigma}}$, trivially $\mathbb{P}(Y_{i+1} \geq y_{i+1}) \leq 1$. Then using $\mu^- = p\sum_{i=0}^{I-1} \tau_i q_i n^2$ and $\mu_{i+1} = 10 \sigma p q_i n^2$ and $\tau_i \geq \frac{1}{2}$, we obtain 
\begin{equation}
\label{eqn: Y(A,B, small y)}
\sum_{0 < i < I: y_{i+1} < \frac{10\mu_{i+1}}{\sqrt{\sigma}}} y_{i+1} \leq \frac{10}{\sqrt{\sigma}}\sum_{i=0}^{I-1} \mu_{i+1} \leq \frac{10}{\sqrt{\sigma}}20\sigma \mu^- \ll y.
\end{equation}

Recalling the notation $Y^I_y=\{{\bf y} =(y_1, \ldots, y_{I}) \in \mathbb N^I: \sum y_i=y\}$, and applying (\ref{eqn: Y(A,B) setup}), (\ref{eqn: Y(A,B) large y}), and (\ref{eqn: Y(A,B, small y)}),

\begin{align*}
    \mathbb{P}(Y(A,B,x) \geq y \text{ and } \mathcal{G}_{\leq I}) &\leq \sum_{{\bf y} \in Y^I_y} \,\, \prod_{i:y_{i+1} \geq \frac{10\mu_{i+1}}{\sqrt \sigma}} \mathbb{P}(Y_{i+1} \geq y_{i+1} | \cap_{j=0}^{i-1} Y_{j+1} \geq y_{j+1} \text{ and } \mathcal{G}_{\leq j+1}) \\
    &\leq \sum_{{\bf y} \in Y^I_y} \sigma^{\frac{y}{2} - o(y)}\\
    &\leq (y+2)^I \sigma^{\frac{y}{4}} \\
    &\leq (2y)^{N^{\beta}} (\log N)^{-\frac{y}{2}} \leq (\log N)^{\frac{-\delta^2 \mu^-}{36}}
\end{align*}
where the last inequality follows from $(\log N)^y \gg y^{I}$ since $y = \Theta(\mu^-)$ and $\mu^- \geq \frac{C \sqrt{\beta}}{2\sqrt{3}} (\log N) n \gg I$. By the union bound
\begin{align*}
    &\mathbb{P}(\exists x, A, B: Y(A,B,x) \geq y \text{ and } \mathcal{G}_{\leq I}) \\
    &\leq N^{2n+1} (\log N)^{\frac{-\delta^2 \mu^-}{36}} \\
    &\leq N^{2n+1}(\log N)^{-\frac{\delta^2 C \sqrt{\beta}}{72\sqrt{3}} (\log N) n}  \\
    &\leq \exp\left((2n+1)\log N-\frac{\delta^2 C \sqrt{\beta}}{72\sqrt{3}} (\log N \log\log N) n\right) \\
    &\leq N^{-\omega(1)}.
    \end{align*}

\section{Appendix}

\subsection{Proof of Lemma \ref{q and pi bounds}}
\begin{enumerate}
    \item To prove (\ref{Psi bound}) first we will prove the following fact for $x \geq e$
    
    \begin{equation} \label{eqn:xineq} \int_0^{\sqrt{\log x} - 1} e^{t^2}dt \leq x \leq \int_0^{\sqrt{\log x} + 1} e^{t^2}dt
    \end{equation}

For the upper bound, notice that the rectangle in $\mathbb{R}^2$ defined by $$R := \{(t,y) | \sqrt{\log x} \leq t \leq \sqrt{\log x} + 1 \text{ , } 0 \leq y \leq x\}$$ clearly has an area of $x$ and falls under the curve $y = e^{t^2}$. Then since for $e^{t^2} \geq 0$ the area of $R$ is no more than $\int_0^{\sqrt{\log x} + 1} e^{t^2}dt$, proving the upper bound.

For the lower bound, for all integers $0 \leq i \leq \sqrt{\log x} - 1$ define the regions $$R_i := \{(t,y) : i-1 \leq t \leq i, 0 \leq y \leq e^{i^2}\}$$ Now let $$S = \cup_{i=0}^{\sqrt{\log x} - 1} R_i$$
and notice that the area of $S$ is at least $\int_{0}^{\sqrt{\log x} - 1} e^{t^2}dt$. Also, the area of $$\text{Area}(S) = \sum_{i=0}^{\sqrt{\log x} - 1} e^{i^2} \leq (\sqrt{\log x} - 1) e^{(\sqrt{\log x} - 1)^2} = x \frac{e(\sqrt{\log x} - 1)}{e^{\sqrt{\log x}}} \leq x$$
where the last inequality holds since $x \geq e$. This completes the lower bound and proves (\ref{eqn:xineq}).

    Recall that $x = \int_0^{\Psi(x)} e^{3t^2}dt$ by the definition of $\Psi(x)$. Since for $t \geq 0$ we have  $e^{3t^2} \geq 0$, and  for $a,b \geq 0$ we have $a \leq b$ iff $\int_0^a e^{3t^2}dt \leq \int_0^b e^{3t^2}dt$. Note that
    $$\int_0^{\sqrt{\frac{\log(\sqrt{3}x)}{3}} \pm \frac{1}{\sqrt{3}}} e^{3t^2}dt = \frac{1}{\sqrt{3}}\int_0^{\sqrt{\log (\sqrt{3}x)} \pm 1} e^{u^2}du$$
    Using (\ref{eqn:xineq}), we obtain
    \begin{equation}
    \label{eqnn:: integral upper and lower bound}
    \int_0^{\sqrt{\frac{\log(\sqrt{3}x)}{3}} - \frac{1}{\sqrt{3}}} e^{3t^2}dt \leq x \leq \int_0^{\sqrt{\frac{\log(\sqrt{3}x)}{3}} + \frac{1}{\sqrt{3}}} e^{3t^2}dt
    \end{equation}
    The bounds on $\Psi(x)$ follow from (\ref{eqnn:: integral upper and lower bound}).

    \item For (\ref{q less than 1}) note $\Psi(x)$ is strictly increasing for $x \geq 0$ and $\Psi'(x)$ is strictly decreasing for $x \geq 0$ since $\Psi'(x) = \exp(-3\Psi^2(x))$. Together with $\Psi'(0) =1$ this gives that $q_i \leq 1$ for all $i$, and notice that since $q_i = \exp (-3\Psi^2(i\sigma))$ then we also get that $q_i \geq 0$.

    \item For (\ref{q and pi bounds pi close to psi}) the first part follows trivially from the definition of $\pi_i$.

    For the second part of this statement, note that since $\Psi(x)$ is strictly increasing and $\Psi'(x)$ is strictly decreasing for $x \geq 0$. Therefore, for all $j < i$,  $$\sigma \Psi'(j\sigma) \geq \Psi((j+1)\sigma) - \Psi(j \sigma) \geq \sigma \Psi'((j+1)\sigma)$$
    Then notice that $\Psi(i \sigma) = \sum_{j=0}^{i-1} (\Psi((j+1)\sigma) - \Psi(j \sigma))$ since $\Psi(0) = 0$, and also recall that $\pi_i = \sigma + \sum_{j=0}^{i-1} \sigma q_j = \sigma + \sum_{j=0}^{i-1} \sigma \Psi'(j\sigma)$. Thus we get 
    \begin{align*}\pi_i - \Psi(i\sigma) &= \left(\sigma + \sum_{j=0}^{i-1} \sigma\Psi'(j\sigma) \right) - \Psi(i\sigma) \\
        &\geq \sigma  + \left(\sum_{j=0}^{i-1} \Psi((j+1)\sigma) - \Psi(j\sigma)\right) - \left(\sum_{j=0}^{i-1} \Psi((j+1)\sigma) - \Psi(j\sigma) \right) = \sigma.\end{align*}
    Similarly
    $$\pi_i - \Psi(i\sigma) = \left(\sigma + \sigma\Psi'(0) - \sigma\Psi'(i\sigma) + \sum_{j=0}^{i-1}\sigma\Psi'((j+1)\sigma) \right) - \Psi(i\sigma) \leq 2\sigma$$
    where the last inequality follows since $\sigma\Psi'(0) = \sigma$ and $\Psi'(i\sigma) \geq 0$.
    
    \item Turning to (\ref{sqrt sigma pi < 1}), by the definition of $\pi_i$ we have that $\pi_i \leq \pi_I$ for all $0 \leq i \leq I$. Then $$\pi_i \leq \sqrt{\frac{\log (\sqrt{3} (I\sigma))}{3}} + 1 \leq \sqrt{\frac{\beta \log N + 1}{3}} \ll \log N = \sigma^{-1/2}.$$

    \item Next for (\ref{q and pi bounds, q times pi less than 1}) we start by noting that for all $x > 0$, $e^{-3x^2}x^2 \leq \frac{1}{2}$ and $e^{-3x^2}x \leq \frac{1}{2}$. Indeed, notice that for $x > 1$ $3x^2 \geq 2\ln x + \ln \frac{1}{2}$, for $\frac{1}{2} \leq x \leq 1$, $-3x^2 \leq \ln \frac{1}{2}$ and $\ln x \leq 0$, and for $0 < x < \frac{1}{2}$, $\ln x \leq \ln \frac{1}{2}$ and $-3x^2 \leq 0$. Then applying this inequality with $x = \Psi(i\sigma)$ together with (\ref{q and pi bounds pi close to psi}) gives
    $$q_i\pi_i^2 \leq e^{-3\Psi^2(i\sigma)}\left(\Psi(i\sigma) + 2\sigma \right)^2 = e^{-3\Psi^2(i\sigma)} \left(\Psi^2(i\sigma) + 4\Psi(i\sigma)\sigma + 4\sigma^2 \right) \leq 1.$$
    
    \item To prove (\ref{bound on q_i - q_i+1}) we will use the fact that for some $s \in [i\sigma, (i+1) \sigma]$, 
    $$\Psi'((i+1)\sigma) = \Psi'(i \sigma) + \Psi''(i\sigma)(\sigma) + \frac{\Psi'''(s)}{2}(\sigma^2)$$
    From the definition of $\Psi(x)$, it can easily be seen that $\Psi''(i\sigma) = -6(\Psi'(i\sigma))^2 \Psi(i \sigma)$ and $\Psi'''(i \sigma) = -6(\Psi'(i\sigma))^3(1 + 2\Psi^2(i\sigma))$. Recalling  $q_i = \Psi'(i\sigma)$ and $\Psi(i \sigma) \in [\pi_i - 2\sigma, \pi_i + 2\sigma]$, we obtain $$q_{i+1} - q_i \in [-6\sigma q_i^2 (\pi_i + 2\sigma) - 3 \sigma^2 q_i^3(1 + 2(\pi_i + 2\sigma)^2),-6\sigma q_i^2 (\pi_i - 2\sigma) - 3 \sigma^2 q_i^3(1 + 2(\pi_i - 2\sigma)^2)]$$
    Now since $6\sigma q_i^2(2\sigma) + 3\sigma^2 q_i^3(1 + 2(\pi + 2\sigma)^2) \leq 16\sigma^2 q_i^2$, we obtain $|(q_{i+1} - q_i) + 6\sigma q_i^2 \pi_i| \leq 16\sigma^2 q_i^2$.

    \item To prove (\ref{eqn:sigmaqN}) note that since for $x \geq e$ we have that $\sqrt{\frac{\log (\sqrt{3}x)}{3}} - \frac{1}{\sqrt{3}} \leq \Psi(x) \leq \sqrt{\frac{\log(\sqrt{3}x)}{3}} + \frac{1}{\sqrt{3}}$ and $\Psi'(x) = \exp(-3\Psi^2(x))$, $\Psi'(x) = \frac{x^{-1 + o(1)}}{e\sqrt{3}}$. Then using the fact that $I = \lceil N^{\beta} \rceil$ and since $\Psi'(x)$ is strictly decreasing we have  $q_i \geq q_I$ for all $0 \leq i \leq I$, and $q_i \geq \frac{N^{-\beta + o(1)}}{e\sqrt{3}}$.

    \item The inequality (\ref{p hat < q}) follows from $$\hat{p}_{e,i} \leq 1 - (1-p)^{6\sqrt{N}q_i(\pi_i + \sqrt{\sigma})} \leq 6\sigma q_i (\pi_i + \sqrt{\sigma}) \leq q_i.$$

    \item Finally for (\ref{q_i - q_i+1 < sigma min...}), note that since $\Psi$ is increasing and $\Psi'$ is decreasing
    \begin{align*}
        |q_i - q_{i+1}| &\leq \sigma \max_{s \in [i\sigma, (i+1)\sigma]} |\Psi''(s)| \\
        &\leq \max_{s \in [i\sigma, (i+1)\sigma]} 6\sigma\left(\Psi'(s)\right)^2\Psi(s) \\
        &\leq 6\sigma (\Psi'(i\sigma) )^2\Psi((i+1)\sigma) \\
        &\leq 6\sigma q_i^2 (\pi_{i+1} - \sigma) \\
        &\leq 6\sigma q_i^2\pi_i.
    \end{align*}
    Then using $q_i, q_i\pi_i \leq 1$ bounds the bounds $q_i - q_{i+1} < 6\sigma q_i$ and $q_i - q_{i+1} < 6\sigma q_i \pi_i$, and for the last bound, notice these imply that $\frac{q_i}{q_{i+1}} < \frac{1}{1-6\sigma} < 2$, so $q_i < 2q_{i+1}$. Then $q_i - q_{i+1} < 12 \sigma q_{i+1}$.
\end{enumerate}

\end{document}